\pdfoutput=1
\RequirePackage{ifpdf}
\ifpdf 
\documentclass[pdftex]{sigma}
\else
\documentclass{sigma}
\fi

\numberwithin{equation}{section}

\newtheorem{Theorem}{Theorem}[section]
\newtheorem{Corollary}[Theorem]{Corollary}
\newtheorem{Proposition}[Theorem]{Proposition}
 { \theoremstyle{definition}
\newtheorem{Example}[Theorem]{Example}
\newtheorem{Remark}[Theorem]{Remark} }

\begin{document}

\allowdisplaybreaks

\renewcommand{\thefootnote}{$\star$}

\newcommand{\arXivNumber}{1502.07256}

\renewcommand{\PaperNumber}{021}

\FirstPageHeading
	
\ShortArticleName{Classes of Bivariate Orthogonal Polynomials}
	
\ArticleName{Classes of Bivariate Orthogonal Polynomials\footnote{This paper is a~contribution to the Special Issue
on Orthogonal Polynomials, Special Functions and Applications.
The full collection is available at \href{http://www.emis.de/journals/SIGMA/OPSFA2015.html}{http://www.emis.de/journals/SIGMA/OPSFA2015.html}}}

\Author{Mourad E.H.~ISMAIL~$^{\dag\ddag}$ and Ruiming ZHANG~$^\S$}
\AuthorNameForHeading{M.E.H.~Ismail and R.~Zhang}

\Address{$^\dag$~Department of Mathematics, King Saud University, Riyadh, Saudi Arabia}

\Address{$^\ddag$~Department of Mathematics, University of Central Florida, Orlando, Florida 32816, USA}
\EmailD{\href{mailto:mourad.eh.ismail@gmail.com}{mourad.eh.ismail@gmail.com}}
\URLaddressD{\url{http://math.cos.ucf.edu/people/ismail-mourad/}}
	
\Address{$^\S$~College of Science, Northwest A{\rm \&}F University, Yangling, Shaanxi 712100, P.R.~China}
\EmailD{\href{mailto:ruimingzhang@yahoo.com}{ruimingzhang@yahoo.com}}

\ArticleDates{Received August 04, 2015, in f\/inal form February 15, 2016; Published online February 24, 2016}	

\Abstract{We introduce a class of orthogonal polynomials in two variables which generalizes the disc polynomials and the 2-$D$ Hermite polynomials. We identify certain interesting members of this class including a one variable generalization of the 2-$D$ Hermite
polynomials and a two variable extension of the Zernike or disc polynomials. We also give $q$-analogues of all these extensions.
In each case in addition to generating functions and three term recursions we provide raising and lowering operators and show that the polynomials are eigenfunctions of second-order partial dif\/ferential or $q$-dif\/ference operators. }
	
	
\Classification{33C50; 33D50; 33C45; 33D45}

\renewcommand{\thefootnote}{\arabic{footnote}}
\setcounter{footnote}{0}

\vspace{-1mm}

\section{Introduction}

One of the earliest orthogonal polynomials are the Legendre polynomials which are orthogonal with respect to $dx=$ (linear measure) on the unit interval. They are the spherical harmonics in $\mathbb{R}^2$. The ultraspherical polynomials
are the spherical harmonics on $\mathbb{R}^m$ and are orthogonal with respect to
 $\frac{\Gamma(\nu+3/2)}{\sqrt{\pi}\Gamma(\nu+1)}(1-x^2)^\nu dx$ on the unit interval. If we
 replace $x$ by $x/\sqrt{\nu}$ and let $\nu \to \infty$ the measure becomes
 $\pi^{-1/2}e^{-x^2} dx$ and the ultraspherical polynomials, properly scaled, tend to Hermite polynomials. The next sequence of polynomials in this hierarchy are the Jacobi polynomials which are orthogonal on the unit interval with respect to
 $\frac{2^{-\alpha-\beta-1}\Gamma(\alpha+\beta+2)}{\Gamma(\alpha+1)\Gamma(\beta+1)}(1-x)^\alpha(1+x)^\beta dx$. The Laguerre polynomials arise when $x$ is replaced by $-1+2x/\alpha$ and $\alpha \to \infty$.

The 2$D$ history is somewhat parallel. In the late 1920's Frits Zernike was working on optical problems involving telescopes and microscopes. He introduced polynomials orthogonal on $|z| \le 1$ with respect to the area measure, so this is the 2$D$ analogue of Legendre polynomials, see~\cite{Zer,Zer:Bri}. The more general disc or Zernike polynomials are orthogonal with respect to $(1-x^2-y^2)^\nu dxdy$ on the unit disc.
Again if we replace $(x,y)$ by $(x/\sqrt{\nu}, y/\sqrt{\nu})$ and let $\nu \to \infty$ this measure, properly normalized becomes $e^{-x^2-y^2} dxdy$ on $\mathbb{R}^2$. The polynomials become the
2$D$-Hermite polynomials introduced by Ito in~\cite{Ito} in a~dif\/ferent way. They are def\/ined by
 \begin{gather*}
 H_{m,n}(z_1, z_2) = \sum_{k=0}^{m \wedge n} \binom{m}{k}
 \binom{n}{k} (-1)^k k! z^{m-k}(\bar z)^{n-k},
 \end{gather*}
 where $m \wedge n=\min\{m,n\}$.
 In 1966 D.R.~Myrick \cite{Myr} considered only the radial part of the polynomials orthogonal with respect to the Jacobi type measure $(x^2+y^2)^\alpha(1-x^2-y^2)^\beta dxdy$ on the unit disc, see also~\cite{Wun3}. Addition theorem for the disc polynomials were found by {\u S}apiro~\cite{Sap} and Koornwinder~\cite{Koo}. Maldonado~\cite{Mal} very brief\/ly considered a class of radial functions which can be used to generate 2$D$ orthogonal polynomials.
 P.~Floris~\cite{Flo} introduced a $q$-analogue of the disc polynomials
 proved an addition theorem for them, see also Floris and Koelink~\cite{Flo:Koe}.

\looseness=-1
 There are several other families of orthogonal polynomials in two and higher variables.
 Koornwinder's
 survey article~\cite{Koo2} covers most of what was known up to the time of its publication in
 1975. In the 19th century Hermite found families of orthogonal and biorthogonal polynomials,
 see \cite{Her1,Her2}, \cite[Chapter~12]{Erd:Mag:Obe:Tri} and \cite{App,App:Kam}.
 The more recent book \cite{Dun:Xu} contains a treatment of the theory and applications of
 orthogonal polynomials in 2 and higher variables. One recent related reference is \cite{Wal}.
The 2$D$-Hermite polynomials
 have may applications to physical problems, see \cite{Ali:Bag:Hon,Cot:Gaz:Gor,Shi,Thi:Hon:Krz,Wun,Wun2}. Mathematical properties of these
 polynomials have been developed in \cite{Gha,Gha2,Int:Int}. A~multilinear
 generating function, of Kibble--Slepian type, is proved in~\cite{Ism4}. In~\cite{Ism:Zha3}
 we gave two $q$-analogues of the
 2$D$-Hermite polynomials and studied the polynomials in great detail. We also
 introduced a dif\/ferent $q$-analogue of the disc polynomials.

 In this paper we f\/irst identify a general class of two variable polynomials whose measure is the
 product of the uniform measure on the circle times a radial measure. This class not only include
 the 2$D$-Hermite polynomials and their $q$-analogues but it also contains the Zernike (or disc)
 polynomials and its $q$-analogues. This will be done in Section~\ref{section2}. One special feature of this class
 is that multiplication by $z$ or $\bar z$ results in a three term recurrence relation. This special
 feature is unlike what happens for general real $2-D$ systems~\cite{Dun:Xu}. In fact the earlier work~\cite{Ism4} motivated Yuan Xu~\cite{Xu} to reformulate the theory of orthogonal polynomials
 in several variables in terms of complex variables.

 In Section~\ref{section3} we give a one parameter extension of the 2$D$-Hermite polynomials. They are
 2$D$ analogues of Laguerre polynomials and will be denoted by $\{Z_{m,n}^{(\beta)}(z, \bar z)\}$.
 We record their def\/inition, orthogonality relation, and the three term recurrence relations
 in Section~\ref{section3}. In Section~\ref{section3} we also derive several dif\/ferential properties of these polynomials. It must be noted that the combinatorics of these polynomials have been explored in \cite{Ism:Zen2}.
Section~\ref{section4} contains a~treatment of a two parameter family of disc polynomials.
They are orthogonal with respect to $(x^2+y^2)^\alpha(1-x^2-y^2)^\beta dxdy$ on the unit disc. They were brief\/ly studied in~\cite{Ism:Zen1} but here we derive new generating functions and study their dif\/ferential-dif\/ference properties and show that they satisfy a Sturm--Liouville system.

The $q$-analogues start with Section~\ref{section5} where we introduce two $q$-analogues of the 2$D$-Laguerre polynomials $\{Z_{m,n}^{(\beta)}(z, \bar z|q)\}$. Both can be considered as a $q$-extension of the 2$D$-Laguerre polynomials, one for $ q \in (0,1)$ and one for $q >1$. As expected the case $q >1$ leads to an indeterminate moment problem while the case $q \in (0,1)$ is a 2$D$-extension of the Wall or little $q$-Laguerre polynomials, \cite{Chi,Koe:Swa}.
A $q$-analogue of the two parameter 2$D$-disc polynomials is introduced in Section~\ref{section6} and many of their properties are developed.
In Section~\ref{section7} we f\/ind polynomial solutions of a~second-order partial dif\/ferential equation.
 Section~\ref{section8} contains a system of biorthogonal polynomials which is a possible 2$D$-analogue of the Askey--Wilson polynomials.

 We shall use the notations
 \begin{gather*}
 (D_{q, z} f)(z) = \frac{f(z) - f(qz)}{(1-q) z},\qquad (\delta_{q,z}f)(z)=z (D_{q, z} f)(z),
 \\
 [\alpha]_q=\frac{1-q^\alpha}{1-q},\qquad m\vee n=\max\{m,n\},\qquad m\wedge n=\min\{m,n\}
 \end{gather*}
 and
 \begin{gather*}
 {\partial}_z=\frac{\partial}{{\partial}_z},\qquad {\delta}_z=z{\partial}_z.
 \end{gather*}

 Throughout the paper we will use $z_1$ and $z_2$. The polynomials we consider here are polynomials of two real variables, so we shall always assume that $z_2 = \bar z_1$.

\section{General constructions}\label{section2}

We now give a concrete construction for general $2D$ systems. Given $N\in \mathbb{N}\cup \{\infty\}$, let $\nu(\theta;N)$ be a probability measure
on the circle $\mathbb{T}=\{e^{i\theta}\,|\,0\le \theta\le 2\pi\}$ of Fourier type
\begin{gather*}
\int_{\mathbb{T}}e^{i(m-n)\theta}d\nu(\theta;N)=\delta_{m,n},\qquad m,n=0,1,\dots,N
\end{gather*}
and $\{\phi_n(r; \alpha)\}$ be a system of orthogonal polynomials satisfying the orthogonality relation
 \begin{gather*}
 \int_0^\infty \phi_m(r; \alpha)\phi_n(r; \alpha) r^\alpha d\mu(r) = \zeta_n(\alpha) \delta_{m,n}, \qquad \alpha \ge 0.
 \end{gather*}
 It is assumed the positive Borel measure $\mu$ does not depend on $\alpha$. Let
 \begin{gather}
 \label{eqphin}
 \phi_n(r;\alpha) = \sum_{j=0}^n c_j(n,\alpha) r^{n-j}, \qquad c_j(n,\alpha) \in \mathbb{R},
 \end{gather}
 and def\/ine polynomials
 \begin{gather}
 \label{eqdeffmn}
 f_{m,n}(z_{1},z_{2};\beta)=\begin{cases}
 z_{1}^{m-n}\phi_{n}(z_{1}z_{2};m-n+\beta), & m\ge n,\\
 f_{n,m}(z_{2},z_{1};\beta), & m<n.
 \end{cases}
 \end{gather}
 This def\/ines the polynomials for all $m$, $n$ and $z_1,z_2\in \mathbb{C}$. It is clear that
 \begin{gather}
 \label{eq:symmetry-condition}
 \overline{f_{m,n}(z, \bar z;\beta)}= f_{n,m}(z, \bar z;\beta), \qquad m<n.
 \end{gather}
From \eqref{eqdeffmn} it is clear that for $m\ge n$
 \begin{gather*}
 z_1f_{m,n}(z_1, z_2;\beta+1)=f_{m+1,n}(z_1,z_2;\beta),
 \end{gather*}
 and
 \begin{gather*}
 -i \frac{\partial}{\partial \theta} f_{m,n}(z, \bar z;\beta) = (m-n) f_{m,n}(z, \bar z;\beta).
 \notag
 \end{gather*}
 Therefore,
 \begin{gather*}
({\delta}_{z_1} - {\delta}_{z_2} )
 f_{m,n}(z_1, z_2;\beta) = (m-n) f_{m,n}(z_1, z_2;\beta), \qquad m \ge n.
 \end{gather*}
 Similarly,
 \begin{gather*}
 \big({\delta}_{q,z_1} -q^{m-n} {\delta}_{q,z_2} \big)f_{m,n}(z_1, z_2;\beta) =[m-n]_q
 f_{m,n}(z_1, z_2;\beta), \qquad m \ge n.
 \end{gather*}

 It must be noted that $f_{m,n}(z_1,z_2)$ is def\/ined for $z_1, z_2 \in \mathbb{C}$ but the next theorem gives an orthogonality relation on $\mathbb{R}^2$. This is exactly analogous to the classical univariate orthogonal polynomials which are def\/ined for $z \in \mathbb{C}$ but their orthogonality is on a subset of $\mathbb{R}$.

 \begin{Theorem} \label{orthfmn}
 Given $N\in \mathbb{N}$, for nonnegative integers $m$, $n$, $s$, $t$ such that $m+t\le N$, $n+s\le N$ the polynomials $\{f_{m,n}(z_1, z_2;\beta)\}$ satisfy the orthogonality relation
 \begin{gather}
 \label{eqorthfmn}
 \int_{\mathbb{R}^2} f_{m,n}(z, \bar z;\beta ) \overline{f_{s,t}(z, \bar z;\beta)}
 d\nu(\theta;N) d\mu\big(r^2;\beta\big)
 = \zeta_{m \wedge n} (|m-n|+\beta) \delta_{m,s} \delta_{n,t},
 \end{gather}
 where $d\mu(r;\beta)=r^{\beta}d\mu(r)$. If $N=\infty$, \eqref{eqorthfmn} holds for all nonnegative integers~$m$, $n$, $s$, $t$.
 \end{Theorem}

 \begin{proof}
 There is no loss of generality in assuming $m \ge n$. Otherwise, we simply apply~\eqref{eq:symmetry-condition} to switch the orders of $m$ and $n$, $s$ and $t$, then consider the consider the complex conjugate of the obtained integral instead.
 By performing the integral over $\theta$ we see that the left-hand side of~\eqref{eqorthfmn}
 equals
 \begin{gather*}
 \int_{\mathbb{R}^2} f_{m,n}(z, \bar z;\beta) \overline{f_{s,t}(z, \bar z;\beta)} d\nu(\theta) d\mu\big(r^2;\beta\big)
 = 0
 \end{gather*}
unless $m+t = n+s$. Let
$\alpha = m-n = s -t$. Therefore the above integral equals
 \begin{gather*}
 \int_0^\infty \sum_{j=0}^n c_j(n,\alpha+\beta) r^{2n + \alpha-2j} \sum_{k=0}^t c_k(t,\alpha+\beta) r^{2t + \alpha-2k} r^{2\beta}d\mu\big(r^2\big) \\
 \qquad{}
 = \int_0^\infty \sum_{j=0}^n c_j(n,\alpha+\beta) r^{n -j} \sum_{k=0}^t c_k(t,\alpha+\beta) r^{t -k}
 r^{\alpha+\beta} d\mu(r) = \zeta_n(\alpha+\beta) \delta_{n,t},
\end{gather*}
and the proof is complete.
\end{proof}

\begin{Remark}
For any positive integer $N$, if we take $\nu(\theta;N)$ as the discrete probability measure used in FFT (fast fourier transform), we get f\/initely many 2D orthogonal polynomials. But in most of our applications we use $N=\infty$ and the circle measure $\frac{d\theta}{2\pi}$ for $d\nu(\theta)$ unless specif\/ically stated.
\end{Remark}

We now come to the three term recurrence relation.
\begin{Theorem}\label{rrfmn}
The polynomials $\{f_{m,n}(z_1, z_2;\beta)\}$ satisfy the three term recurrence relation
\begin{gather}
z_2 f_{m+1,n}(z_1, z_2;\beta)
= \frac{c_0(n, m-n+1+\beta)} {c_0(n+1, m-n+\beta )} f_{m+1, n+1}(z_1, z_2;\beta)\nonumber\\
\qquad{}
- \frac{c_0(n, m-n+1+\beta)c_{n+1}(n+1, m-n+\beta)} {c_0(n+1, m-n+\beta) c_n(n, m-n+\beta)} f_{m, n}(z_1, z_2;\beta), \qquad m \ge n.\label{eq3trrG2D}
\end{gather}
\end{Theorem}

\begin{proof}
Since the polynomials
$\{ \phi_n(r;\alpha) \}$ are orthogonal on $[0, \infty)$, their zeros are in $(0, \infty)$. Hence $ \phi_n(0;\alpha) \ne 0$ for all $n$ and all $\alpha\ge 0$. From the Christof\/fel formula \cite[Theorem~2.7.1]{Ism} we see that
\begin{gather*}
r \phi_n(r;\alpha+1)
= \frac{c_0(n, \alpha+1)} {c_0(n+1, \alpha) \phi_n(0;\alpha)}
[ \phi_{n+1}(r;\alpha) \phi_n(0;\alpha) - \phi_n(r;\alpha) \phi_{n+1}(0;\alpha)],
\end{gather*}
that is
\begin{gather}
\label{eqChrisFor}
 r \phi_n(r;\alpha+1) = \frac{c_0(n, \alpha+1)} {c_0(n+1, \alpha)} \phi_{n+1}(r;\alpha) -
\frac{c_0(n, \alpha+1)c_{n+1}(n+1, \alpha)} {c_0(n+1, \alpha) c_n(n, \alpha)} \phi_n(r;\alpha).
\end{gather}
Now take $\alpha = m-n+\beta$, replace $r$ by $z_1 z_2$ then multiply both sides by ${z_1}^{m-n}$ to estab\-lish~\eqref{eq3trrG2D}.
\end{proof}

One can also prove that \cite{Ism:Zen1}
\begin{gather}
\label{eqrecrel2}
z_1 f_{m,n}(z_1, z_2;\beta) - \frac{c_0(n, m-n+\beta)}{c_0(n,m+1-n+\beta)} f_{m+1,n}(z_1, z_2;\beta) = u_{m,n} f_{m, n-1}(z_1, z_2;\beta),
\end{gather}
where{\samepage
\begin{gather*}
u_{m,n} = \frac{ c_0(n, \alpha+1) c_1(n,\alpha) - c_0(n, \alpha) c_1(n,1+ \alpha)}
 {c_0(n-1, 1+\alpha)c_0(n, 1+\alpha)}
\end{gather*}
and $\alpha=m-n+\beta$.}

It must be noted that the 2$D$-polynomials may have recurrence relations containing more than three terms~\cite{Dun:Xu}. It is surprising that the polynomials in the special class we are considering have three term recursions.
Another important point is that the def\/inition of the polynomials~$f_{m,n}$ indicates that it has either a trivial zero at $z =0$ or the zeros will lie on circles whose radii are the square roots of the zeros of $\phi_n(r;m-n+\beta)$. The zeros of $\phi_n(r;m-n+\beta)$ are positive, real and simple.

Let
\begin{gather}
r\phi_n(r; \alpha) = a_n(\alpha) \phi_{n+1}(r; \alpha)+ c_n(\alpha) \phi_n(r; \alpha)+ b_n(\alpha) \phi_{n-1}(r; \alpha)
\label{eq3trrphi}
\end{gather}
be the three term recurrence relation satisf\/ied by the polynomials $\{\phi_n(r; \alpha)\}$,
for some real numbers $a_{n}(\alpha)$, $b_{n}(\alpha)$, $c_{n}(\alpha)$. By matching the corresponding
coef\/f\/icients we f\/ind that
\begin{gather}
a_{n}(\alpha) = \frac{c_{0}(n,\alpha)}{c_{0}(n+1,\alpha)},\label{eq:eqan}\\
b_{n}(\alpha) = \frac{c_{0}(n,\alpha)c_{2}(n,\alpha)-c_{1}(n,\alpha)^{2}}{c_{0}(n-1,\alpha)c_{0}(n,\alpha)}
 - \frac{c_{0}(n,\alpha)c_{2}(n+1,\alpha)-c_{1}(n,\alpha)c_{1}(n+1,\alpha)}{c_{0}(n-1,\alpha)c_{0}(n+1,\alpha)},\label{eqbn} \\
c_{n}(\alpha) = \frac{c_{1}(n,\alpha)}{c_{0}(n,\alpha)}-\frac{c_{1}(n+1,\alpha)}{c_{0}(n+1,\alpha)}.\label{eqcn}
\end{gather}

 Making use of \eqref{eqphin} and \eqref{eqdeffmn} we conclude that the recursion \eqref{eq3trrphi} becomes
\begin{gather}
 [z_1 z_2 - c_n(m-n+\beta)] f_{m,n}(z_1, z_2;\beta)\nonumber\\
\qquad {}= a_n(m-n+\beta) f_{m+1,n+1}(z_1, z_2;\beta) + b_n(m-n+\beta) f_{m-1,n-1}(z_1,z_2;\beta).\label{eqrecrel2-1-2}
\end{gather}

\begin{Example}
Consider the case of the 2-$D$ Hermite polynomials when the coef\/f\/icients in \eqref{eqdeffmn} are given by
 \begin{gather*}
 c_j(n,\alpha) = \frac{n!(n+\alpha)!(-1)^j}{j!(n-j)!(n+\alpha-j)!}.
 \end{gather*}
 The recursion \eqref{eq3trrG2D} leads to the second three term recurrence relation
 \begin{gather*}
 z_1 H_{m,n}(z_1, z_2) = nH_{m,n-1}(z_1, z_2)+ H_{m+1,n}(z_1, z_2), \\
 z_2 H_{m,n}(z_1, z_2|q) = m H_{m-1,n}(z_1, z_2)+ H_{m,n+1}(z_1, z_2),
\end{gather*}
 while the f\/irst can be proved by direct computation.

 Similarly the case when the coef\/f\/icients in \eqref{eqdeffmn} are given by
 \begin{gather*}
 c_j(n,\alpha) = \frac{(q;q)_n(q;q)_{n+\alpha}(-1)^j}{(q;q)_j(q;q)_{n-j}(q;q)_{n+\alpha-j}}
 q^{\binom{j}{2}}
 \end{gather*}
 leads to the recurrence relations for this $q$-analogue, see \cite{Ism:Zha3}.
 \end{Example}

\begin{Remark}If we demand that both \eqref{eqdeffmn} and \eqref{eq:symmetry-condition} are valid for all $m$, $n$, $z$, then we must have
\begin{gather*}
c_{j}(n,\alpha)=\binom{n}{j}\frac{a_{j}(n,\alpha)}{(n+\alpha-j)!}
\end{gather*}
for the classical cases, and
\begin{gather}
c_{j}(n,\alpha)=\binom{n}{j}_q\frac{a_{j}(n,\alpha|q)}{(q;q)_{n+\alpha-j}}\label{eq2.21}
\end{gather}
for the $q$-analogues, where $c_{j}(n,\alpha)$ are the coef\/f\/icients in~\eqref{eqdeffmn}. These factorizations imply that $\phi_{n}(r;\alpha)$
must be a generalization of Laguerre orthogonal polynomials.
However, in most of the cases, \eqref{eqdeffmn} is only valid for $m\geq n$, whereas~\eqref{eq:symmetry-condition} acts as a kind of analytical continuation to the case $m<n$.
\end{Remark}

\begin{Remark}
In view of the def\/ining equation~\eqref{eqdeffmn} the question of f\/inding the large $m$, $n$ asymptotics of the two variate polynomials $f_{m,n}$ is equivalent to f\/inding the large~$m$, $n$ behavior of $\phi_n(r, m-n+\beta)$. It will be of interest to carry out this program at least for some special systems, including the Freud type polynomials orthogonal with respect to
$x^\alpha \exp(-p(x))$, where~$p$ is a polynomial with positive leading term.
 \end{Remark}

 We now consider an inverse to \eqref{eqChrisFor}.
 \begin{Proposition}
\label{prop:ismzen1} Given a nonnegative measure on $(0,\infty)$
\begin{gather*}
d\mu(x;\alpha)=x^{\alpha}d\nu(x),\quad x\ge0,\alpha>-1
\end{gather*}
 such that
\begin{gather*}
\int_{0}^{\infty}x^{n}d\mu(x;\alpha)<\infty,\qquad n=0,1,\dots.
\end{gather*}
 Let $\{\phi_{n}(x;\alpha)\}_{n=0}^{\infty}$ satisfy
\begin{gather*}
\int_{0}^{\infty}\phi_{m}(x;\alpha)\phi_{n}(x;\alpha)d\mu(x;\alpha)=\zeta_{n}(\alpha)\delta_{m,n}.
\end{gather*}
 Then
\begin{gather}
\phi_{n}(x;\alpha)-a_{n}(\alpha)\phi_{n}(x;\alpha+1)=b_{n}(\alpha)\phi_{n-1}(x;\alpha+1),\label{eq:ismzen-4}
\end{gather}
 where
\begin{gather*}
\phi_{n}(x;\alpha)=\sum_{j=0}^{n}c_{j}(n,\alpha)x^{n-j}, 
\\
b_{n}(\alpha) = \frac{c_{0}(n,\alpha+1)c_{1}(n,\alpha)-c_{0}(n,\alpha)c_{1}(n,\alpha+1)}{c_{0}(n-1,\alpha+1)c_{0}(n,\alpha+1)},\\ 
a_{n}(\alpha) = \frac{c_{0}(n,\alpha)}{c_{0}(n,\alpha+1)}.
\end{gather*}
Furthermore,
\begin{gather}
 \int_{0}^{\infty}x\phi_{n}(x;\alpha)\phi_{n-1}(x;\alpha+1)d\mu(x;\alpha) =b_{n}(\alpha)\zeta_{n-1}(\alpha+1).
\label{eq:ismzen-8}
\end{gather}
\end{Proposition}

\begin{proof}
It is clear that $\phi_{n}(x;\alpha)-a_{n}(\alpha)\phi_{n}(x;\alpha+1)$
is a polynomial of degree at most $n-1$, then
\begin{gather*}
\phi_{n}(x;\alpha)-a_{n}(\alpha)\phi_{n}(x;\alpha+1)=\sum_{j=0}^{n-1}\alpha_{j}\phi_{j}(x;\alpha+1),
\end{gather*}
where
\begin{gather*}
\alpha_{j}\zeta_{j}(\alpha+1) = \int_{0}^{\infty} \{ \phi_{n}(x;\alpha)-a_{n}(\alpha)\phi_{n}(x;\alpha+1) \} \phi_{j}(x;\alpha+1)d\mu(x;\alpha+1)\\
\hphantom{\alpha_{j}\zeta_{j}(\alpha+1)}{}
 = \int_{0}^{\infty}\phi_{n}(x;\alpha)\phi_{j}(x;\alpha+1)d\mu(x;\alpha+1)\\
\hphantom{\alpha_{j}\zeta_{j}(\alpha+1)}{}
 = \int_{0}^{\infty}x\phi_{n}(x;\alpha)\phi_{j}(x;\alpha+1)d\mu(x;\alpha).
\end{gather*}
Clearly $\alpha_{j}=0$ for $0\le j\le n-2$ and $\alpha_{n-1}=b_{n}(\alpha)$.
Then $b_{n}(\alpha)$ is obtained by matching the leading coef\/f\/icients
of two sides.
\end{proof}

\begin{Corollary}
\label{cor:ismzen2} Under the assumptions of Proposition~{\rm \ref{prop:ismzen1}},
let
\begin{gather}
\phi_{n}(x;\alpha+1)=\sum_{j=0}^{n}\lambda_{j}(n,\alpha)\phi_{j}(x;\alpha).\label{eq:ismzen-9}
\end{gather}
Then
\begin{gather}
\lambda_{j}(n,\alpha) = \frac{(-1)^{n-j}}{a_{j}(\alpha)}\prod_{k=0}^{n-j-1}\frac{b_{n-k}(\alpha)}{a_{n-k}(\alpha)},\qquad 0\le j\le n-1,\label{eq:ismzen-10}\\
\lambda_{n}(n,\alpha) = \frac{1}{a_{n}(\alpha)}\label{eq:ismzen-11}
\end{gather}
and
\begin{gather}
\frac{\zeta_{n}(\alpha)}{\zeta_{0}(\alpha+n)} = \prod_{j=0}^{n-1}\frac{b_{n-j}(\alpha+j)c_{0}(n-j,\alpha+j)}{a_{n-j-1}(\alpha+j)c_{0}(n-j-1,\alpha+j)}.\label{eq:ismzen-12}
\end{gather}
Furthermore, we have
\begin{gather*}
d_{n}(\alpha)\phi_{n}(x;\alpha)=\phi_{n}(x;\alpha+1)-x\sum_{j=0}^{n-1}f_{j}(n,\alpha)\phi_{j}(x;\alpha+1)
\end{gather*}
and
\begin{gather*}
\phi_{n}(x;\alpha)=\phi_{n}(0;\alpha)-x\sum_{j=0}^{n-1}g_{j}(n,\alpha)\phi_{j}(x;\alpha+1), 
\end{gather*}
where
\begin{gather*}
d_{n}(\alpha) = \frac{c_{n}(n,\alpha+1)}{c_{n}(n,\alpha)},
\\
f_{k}(n,\alpha)=\sum_{j=0}^{k}\frac{\lambda_{j}(n,\alpha)\lambda_{j}(k,\alpha)\zeta_{j}(\alpha+1)}{\zeta_{k}(\alpha+1)},\qquad 0\le k\le n-1,
\\
g_{k}(n,\alpha)=\frac{(-1)^{k}\zeta_{0}(\alpha)c_{0}(0,\alpha+1)c_{n}(n,\alpha)}{\zeta_{k}(\alpha+1)}
\prod_{j=0}^{k-1}\frac{b_{k-j}(\alpha)}{a_{k-j}(\alpha)}.
\end{gather*}
 \end{Corollary}

\begin{proof}
From \eqref{eq:ismzen-9} and \eqref{eq:ismzen-4} we get
\begin{gather*}
\phi_{n}(x;\alpha)-a_{n}(\alpha)\sum_{j=0}^{n}\lambda_{j}(n,\alpha)\phi_{j}(x;\alpha)=b_{n}(\alpha)\sum_{j=0}^{n-1}\lambda_{j}(n-1,\alpha)\phi_{j}(x;\alpha).
\end{gather*}
Then
\begin{gather*}
1-\lambda_{n}(n,\alpha)a_{n}(\alpha)  =  0,
\end{gather*}
and
\begin{gather*}
-a_{n}(\alpha)\lambda_{j}(n,\alpha)=b_{n}(\alpha)\lambda_{j}(n-1,\alpha)
\end{gather*}
for $0\le j\le n-1$, which lead to \eqref{eq:ismzen-11} and \eqref{eq:ismzen-10}
respectively.

From \eqref{eq:ismzen-8} we get \eqref{eq:ismzen-9} we get
\begin{gather*}
b_{n}(\alpha)\zeta_{n-1}(\alpha+1)   =   \lambda_{n-1}(n-1,\alpha)\int_{0}^{\infty}x\phi_{n}(x;\alpha)\phi_{n-1}(x;\alpha)d\mu(x;\alpha)\\
\hphantom{b_{n}(\alpha)\zeta_{n-1}(\alpha+1) }{} =   \frac{1}{a_{n-1}(\alpha)}\int_{0}^{\infty}x\phi_{n}(x;\alpha)\phi_{n-1}(x;\alpha)d\mu(x;\alpha).
\end{gather*}
Since $\{\phi_{n}(x;\alpha)\}_{n=0}^{\infty}$ are orthogonal polynomials
with respective to the measure $d\mu(x;\alpha)$, hence
\begin{gather*}
x\phi_{n-1}(x;\alpha)=\eta_{n}\phi_{n}(x;\alpha)+\delta_{n}\phi_{n-1}(x;\alpha)+\epsilon_{n}\phi_{n-2}(x;\alpha),
\end{gather*}
 where
\begin{gather*}
\eta_{n}=\frac{c_{0}(n-1,\alpha)}{c_{0}(n,\alpha)}.
\end{gather*}
Then
\begin{gather*}
b_{n}(\alpha)\zeta_{n-1}(\alpha+1)=\frac{\eta_{n}\zeta_{n}(\alpha)}{a_{n-1}(\alpha)},
\end{gather*}
and
\begin{gather*}
\zeta_{n}(\alpha) = \frac{b_{n}(\alpha)}{a_{n-1}(\alpha)}\frac{c_{0}(n,\alpha)}{c_{0}(n-1,\alpha)}\zeta_{n-1}(\alpha+1)\\
\hphantom{\zeta_{n}(\alpha)}{} = \zeta_{0}(\alpha+n)\prod_{j=0}^{n-1}\frac{b_{n-j}(\alpha+j)}{a_{n-j-1}(\alpha+j)}\frac{c_{0}(n-j,\alpha+j)}{c_{0}(n-j-1,\alpha+j)},
\end{gather*}
which gives \eqref{eq:ismzen-12}.

Let
\begin{gather*}
\phi_{n}(x;\alpha+1)-d_{n}(\alpha)\phi_{n}(x;\alpha)=x\sum_{j=0}^{n-1}f_{j}(n,\alpha)\phi_{j}(x;\alpha+1).
\end{gather*}
 For $0\le k\le n-1$ we have
\begin{gather*}
 \int_{0}^{\infty}\left\{ \phi_{n}(x,\alpha+1)-d_{n}(\alpha)\phi_{n}(x,\alpha)\right\} \phi_{k}(x,\alpha+1)d\mu(x;\alpha)\\
\qquad{} =\int_{0}^{\infty}\left\{ x\sum_{j=0}^{n-1}f_{j}(n,\alpha)\phi_{j}(x,\alpha+1)\right\} \phi_{k}(x,\alpha+1)d\mu(x;\alpha)\\
\qquad{} =\int_{0}^{\infty}\left\{ \sum_{j=0}^{n-1}f_{j}(n,\alpha)\phi_{j}(x,\alpha+1)\right\} \phi_{k}(x,\alpha+1)d\mu(x;\alpha+1)\\
\qquad{} =f_{k}(n,\alpha)\zeta_{k}(\alpha+1),
\end{gather*}
which is
\begin{gather*}
f_{k}(n,\alpha)\zeta_{k}(\alpha+1) = \int_{0}^{\infty}\phi_{n}(x,\alpha+1)\phi_{k}(x,\alpha+1)d\mu(x;\alpha)\\
 \hphantom{f_{k}(n,\alpha)\zeta_{k}(\alpha+1)}{} = \sum_{j=0}^{k}\lambda_{j}(n,\alpha)\lambda_{j}(k,\alpha)\zeta_{j}(\alpha+1),
\end{gather*}
 or
\begin{gather*}
f_{k}(n,\alpha)=\sum_{j=0}^{k}\frac{\lambda_{j}(n,\alpha)\lambda_{j}(k,\alpha)\zeta_{j}(\alpha+1)}{\zeta_{k}(\alpha+1)}.
\end{gather*}
 On the other hand, let
\begin{gather*}
\phi_{n}(x;\alpha)=\phi_{n}(0;\alpha)-x\sum_{j=0}^{n-1}g_{j}(n,\alpha)\phi_{j}(x;\alpha+1),
\end{gather*}
then, for $0\le k\le n-1$, we have
\begin{gather*}
 \int_{0}^{\infty}\left\{ \phi_{n}(x;\alpha)-\phi_{n}(0;\alpha)\right\} \phi_{k}(x,\alpha+1)d\mu(x;\alpha)\\
\qquad{} =-\int_{0}^{\infty}\left\{ x\sum_{j=0}^{n-1}g_{j}(n,\alpha)\phi_{j}(x;\alpha+1)\right\} \phi_{k}(x,\alpha+1)d\mu(x;\alpha)\\
\qquad{} =-\int_{0}^{\infty}\left\{ \sum_{j=0}^{n-1}g_{j}(n,\alpha)\phi_{j}(x;\alpha+1)\right\} \phi_{k}(x,\alpha+1)d\mu(x;\alpha+1)\\
\qquad{} =-g_{k}(n,\alpha)\zeta_{k}(\alpha+1),
\end{gather*}
that is
\begin{gather*}
 g_{k}(n,\alpha)\zeta_{k}(\alpha+1) = \int_{0}^{\infty}\phi_{n}(0;\alpha)\phi_{k}(x,\alpha+1)d\mu(x;\alpha)\\
\hphantom{g_{k}(n,\alpha)\zeta_{k}(\alpha+1)}{}
= \phi_{n}(0;\alpha)\lambda_{0}(k,\alpha)\phi_{0}(0,\alpha)\zeta_{0}(\alpha)
 = c_{n}(n,\alpha)c_{0}(0,\alpha)\zeta_{0}(\alpha)\lambda_{0}(k,\alpha)\\
\hphantom{g_{k}(n,\alpha)\zeta_{k}(\alpha+1)}{}
 = (-1)^{k}\zeta_{0}(\alpha)c_{0}(0,\alpha+1)c_{n}(n,\alpha)\prod_{j=0}^{k-1}\frac{b_{k-j}(\alpha)}{a_{k-j}(\alpha)},
\end{gather*}
that is
\begin{gather*}
g_{k}(n,\alpha)=\frac{(-1)^{k}\zeta_{0}(\alpha)c_{0}(0,\alpha+1)c_{n}(n,\alpha)}{\zeta_{k}(\alpha+1)}
\prod_{j=0}^{k-1}\frac{b_{k-j}(\alpha)}{a_{k-j}(\alpha)}.\tag*{\qed}
\end{gather*}
 \renewcommand{\qed}{}
\end{proof}

\begin{Corollary}
\label{cor:ismzen3} Let $\phi_{n}(x,\alpha)$, $a_{n}(\alpha)$, $b_{n}(\alpha)$
as in Proposition~{\rm \ref{prop:ismzen1}}, for $m\ge n-1$ the equa\-tion~\eqref{eqrecrel2} is in the form
\begin{gather}
z_{1}f_{m,n}(z_{1},z_{2})-v_{m,n}f_{m+1,n}(z_{1},z_{2})=u_{m,n}f_{m,n-1}(z_{1},z_{2}),\label{eq:ismzen-18}
\end{gather}
 where
\begin{gather*}
u_{m,n} = b_{n}(m-n),\qquad v_{m,n}=a_{n}(m-n).
\end{gather*}
 \end{Corollary}

\begin{proof}
For $m\ge n-1$, applying the def\/inition $f_{m,n}(z_{1},z_{2})=z_{1}^{m-n}\phi_{n}(z_{1}z_{2};m-n)$, $m\ge n$
to get~\eqref{eq:ismzen-18}.
\end{proof}

We next consider dif\/ferential or $q$-dif\/ference equations satisf\/ied by $f_{m,n}(z_1, z_2)$.

\begin{Theorem}\label{theorem2.4}
Assume that $\phi_{n}(r;\alpha)$ satisfies the second-order differential
equation
\begin{gather*}
A_{n}(r,\alpha)\delta_{r}^{2}f+B_{n}(r,\alpha)\delta_{r}f+C_n(r,\alpha)f=0,
\end{gather*}
where
\begin{gather*}
\delta_{r}f=r\frac{\partial}{\partial r}f.
\end{gather*}
Then $f_{m,n}(z_{1},z_{2};\beta)$, for $\alpha=m-n+\beta$ and $m\ge n$, satisfies the second-order partial
differential equations
\begin{gather*}
A_{n}(z_{1}z_{2},\alpha)\delta_{z_{2}}^{2}f+B_{n}(z_{1}z_{2},\alpha)\delta_{z_{2}}f+C_n(z_{1}z_{2},\alpha)f=0,
\\
 A_{n}(z_{1}z_{2},\alpha)\delta_{z_{1}}^{2}f+\left\{ B_{n}(z_{1}z_{2},\alpha)-2\alpha A_{n}(z_{1}z_{2},\alpha)\right\} \delta_{z_{1}}f\\
\qquad{} +\left\{ \alpha^{2}A_{n}(z_{1}z_{2},\alpha)-\alpha B_{n}(z_{1}z_{2},\alpha)+C_n(z_{1}z_{2},\alpha)\right\} f=0.
\end{gather*}
Similarly, assume that $\phi_{n}(r;\alpha)$ satisfies the second
order $q$-difference equation
\begin{gather*}
A_{n}(r,\alpha)\delta_{q,r}^{2}f+B_{n}(r,\alpha)\delta_{q,r}f+C_n(r,\alpha)f=0,
\end{gather*}
where
\begin{gather*}
\delta_{q,r}f=rD_{q,r}f.
\end{gather*}
Then $f_{m,n}(z_{1},z_{2})$, for $m\ge n$ and $\alpha=m-n+\beta$, satisfies the $q$-partial
difference equations
\begin{gather*}
A_{n}(z_{1}z_{2},\alpha)\delta_{q,z_{2}}^{2}f+B_{n}(z_{1}z_{2},\alpha)\delta_{q,z_{2}}f+C_n(z_{1}z_{2},\alpha)f=0,
\\
 A_{n}(z_{1}z_{2},\alpha)\delta_{q,z_{1}}^{2}f+\big\{ q^{\alpha}B_{n}(z_{1}z_{2},\alpha)-2[\alpha]_{q}A_{n}(z_{1}z_{2},\alpha)\big\} \delta_{q,z_{1}}f\\
\qquad{} + \big\{ [\alpha]_{q}^{2}A_{n}(z_{1}z_{2},\alpha)-[\alpha]_{q}q^{\alpha}B_{n}(z_{1}z_{2},\alpha)+q^{2\alpha}C_n(z_{1}z_{2},\alpha)\big\} f=0.
\end{gather*}
\end{Theorem}

\begin{proof}
Observe that
\begin{gather*}
\phi_{n}(z_{1}z_{2}) = z_{1}^{-\alpha}f_{m,n}(z_{1},z_{2}), \\
\delta_{z_{1}z_{2}}\phi_{n}(z_{1}z_{2}) = z_{1}^{-\alpha}\delta_{z_{2}}f_{m,n}(z_{1},z_{2})
 = z_{1}^{-\alpha} \{ \delta_{z_{1}}-\alpha \} f\\
\delta_{z_{1}z_{2}}^{2}\phi_{n}(z_{1}z_{2}) = z_{1}^{-\alpha}\delta_{z_{2}}^{2}f_{m,n}(z_{1},z_{2})
 = z_{1}^{-\alpha}\big\{ \delta_{z_{1}}^{2}-2\alpha\delta_{z_{1}}+\alpha^{2}\big\} f\\
\delta_{q,z_{1}z_{2}}\phi_{n}(z_{1}z_{2}) = z_{1}^{-\alpha}\delta_{q,z_{2}}f_{m,n}(z_{1},z_{2})
 = q^{-\alpha}z_{1}^{-\alpha} \{ \delta_{q,z_{1}}-[\alpha]_{q} \} f,\\
\delta_{q,z_{1}z_{2}}^{2}\phi_{n}(z_{1}z_{2}) = z_{1}^{-\alpha}\delta_{q,z_{2}}^{2}f_{m,n}(z_{1},z_{2})
 = q^{-2\alpha}z_{1}^{-\alpha}\big\{ \delta_{q,z_{1}}^{2}-2[\alpha]_{q}\delta_{q,z_{1}}+[\alpha]_{q}^{2}\big\} f.
\end{gather*}
Then the proof of theorem follows from substituting the required combinations
into the corresponding dif\/ferential equation or $q$-dif\/ference equation
for $\phi(r,\alpha)$.
\end{proof}

A simple application of Fubini's theorem establishes the following theorem concerning generating functions.
\begin{Theorem}
\label{theorem:f-m-n-generating-function}
For $z_{1},z_{2},u,v\in\mathbb{C}$ such that
\begin{gather*}
\sum_{j=0}^{\infty}\frac{\left|uz_{1}\right|^{j}}{j!}\left|g_{j}(z_{1}z_{2},uv)\right|<\infty,\qquad 
\sum_{j=0}^{\infty}\frac{\left|uz_{1}\right|^{j}}{(q;q)_j}\left|g_{j}(z_{1}z_{2},uv)\right|<\infty,
\end{gather*}
where
\begin{gather*}
g_{\alpha}(x,z)=\sum_{k=0}^{\infty}\varphi_{k}(x;\alpha)z^{k}.
\end{gather*}
Then
\begin{gather}
\sum_{m\ge n\ge0}\frac{u^{m}v^{n}}{(m-n)!}f_{m,n}(z_{1},z_{2})=\sum_{j=0}^{\infty}\frac{(uz_{1})^{j}}{j!}g_{j}(z_{1}z_{2},uv),
\label{eq:f-m-n-generating-function}\\
\sum_{m\ge n\ge0}\frac{u^{m}v^{n}}{(q;q)_{m-n}}f_{m,n}(z_{1},z_{2}|q)
=\sum_{j=0}^{\infty}\frac{(uz_{1})^{j}}{(q;q)_j}g_{j}(z_{1}z_{2},uv).\nonumber 
\end{gather}
Similarly, for $z_{1},z_{2},u,v\in\mathbb{C}$ such that
\begin{gather*}
\sum_{j=0}^{\infty}\left|uz_{1}\right|^{j}\left|g_{j}(z_{1}z_{2},uv)\right|<\infty,
\end{gather*}
then
\begin{gather*}
\sum_{m\ge n\ge0}u^{m}v^{n}f_{m,n}(z_{1},z_{2})=\sum_{j=0}^{\infty}(uz_{1})^{j}g_{j}(z_{1}z_{2},uv),\\ 
\sum_{m\ge n\ge0}u^{m}v^{n}f_{m,n}(z_{1},z_{2}|q)=\sum_{j=0}^{\infty}(uz_{1})^{j}g_{j}(z_{1}z_{2},uv).
\end{gather*}
\end{Theorem}

\begin{Theorem}
\label{theorem:f-m-n-roderigues}Assume that
\begin{gather*}
w_{\alpha+\beta}(x)=x^{\alpha}w_{\beta}(x)
\end{gather*}
and
\begin{gather*}
\phi_{n}(x;\alpha)=\frac{1}{w_{\alpha}(x)}\partial_{x}^{n}\left(w_{\alpha}(x)x^{n}\right),\qquad n=0,1,\dots.
\end{gather*}
Then for $m\ge n$ we have
\begin{gather}
f_{m,n}(z_{1},z_{2};\beta)=\left(\frac{z_{1}}{z_{2}}\right)^{m-n}
\frac{\partial_{z_{2}}^{n}\left(w_{\beta}(z_{1}z_{2})z_{2}^{m}\right)}{w_{\beta}(z_{1}z_{2})},\label{eq:f-m-n-roderigues-2}
\\
f_{m+1,n+1}(z_{1},z_{2};\beta)=\frac{\partial_{z_{2}}\left(\left(\frac{z_{2}}{z_{1}}\right)^{m+1-n}w_{\beta}(z_{1}z_{2})
f_{m+1,n}(z_{1},z_{2};\beta)\right)}{\left(\frac{z_{2}}{z_{1}}\right)^{m-n}w_{\beta}(z_{1}z_{2})},\label{eq:f-m-n-shift-n-down-2}
\\
f_{m,n}(z_{1},z_{2};\beta)=\frac{\partial_{z_{1}}^{n}\left(w_{\beta}(z_{1}z_{2})z_{1}^{m}\right)}{w_{\beta}(z_{1}z_{2})},\label{eq:f-m-n-roderigues-1}
\\
f_{m+1,n+1}(z_{1},z_{2};\beta)=\frac{\partial_{z_{1}}\left(w_{\beta}(z_{1}z_{2})f_{m+1,n}(z_{1},z_{2};\beta)\right)}{w_{\beta}(z_{1}z_{2})}
\label{eq:f-m-n-shift-n-down-1}
\end{gather}
and
\begin{gather}
\partial_{z_{2}}\left(\left(\frac{z_{2}}{z_{1}}\right)^{m+1-n}w_{\beta}(z_{1}z_{2})f_{m+1,n}(z_{1},z_{2};\beta)\right)
=\partial_{z_{1}}\left(w_{\beta}(z_{1}z_{2})f_{m+1,n}(z_{1},z_{2};\beta)\right).\label{eq:f-m-n-pde-from-n}
\end{gather}
Furthermore, if for some nonnegative integers $s$, $t$ such that
\begin{gather*}
\frac{w_{\beta}'(x)}{w_{\beta}(x)}=\frac{\sum\limits_{j=0}^{s}a_{j}x^{j}}{\sum\limits_{j=0}^{t}b_{j}x^{j}},
\end{gather*}
where $\{a_{j}\}_{j=0}^{s}$ and $\{b_{j}\}_{j=0}^{t}$ are certain
real numbers. Then form $m\ge n$ we have
\begin{gather}
 \sum_{j=0}^{t}b_{j}z_{2}^{j}\frac{\partial_{z_{2}}\left(w_{\beta}(z_{1}z_{2})f_{m+j-1,n}(z_{1},z_{2};\beta)\right)}{w_{\beta}(z_{1}z_{2})}
= \sum_{j=0}^{s}a_{j}z_{2}^{j}f_{m+j,n}(z_{1},z_{2};\beta),
\label{eq:f-m-n-differential-m-recurrence-1}
\\ \sum_{j=0}^{t}b_{j}\frac{\partial_{z_{1}}\left(z_{1}^{n-m-j+1}w_{\beta}(z_{1}z_{2})f_{m+j-1,n}(z_{1},z_{2};\beta)\right)}
{z_{1}^{n-m-j}z_{2}^{1-j}w_{\beta}(z_{1}z_{2})}
= \sum_{j=0}^{s}a_{j}z_{2}^{j}f_{m+j,n}(z_{1},z_{2};\beta)
\label{eq:f-m-n-differential-m-recurrence-2}
\end{gather}
and
\begin{gather}
 \sum_{j=0}^{t}b_{j}z_{2}^{j}\frac{\partial_{z_{2}}\left(w_{\beta}(z_{1}z_{2})f_{m+j-1,n}(z_{1},z_{2};\beta)\right)}{w_{\beta}(z_{1}z_{2})}\nonumber\\
\qquad{} = \sum_{j=0}^{t}b_{j}\frac{\partial_{z_{1}}\left(z_{1}^{n-m-j+1}w_{\beta}(z_{1}z_{2})f_{m+j-1,n}(z_{1},z_{2};\beta)\right)}{z_{1}^{n-m-j}z_{2}^{1-j}w_{\beta}(z_{1}z_{2})}.
\label{eq:f-m-n-pde-from-m}
\end{gather}
 \end{Theorem}

\begin{proof}
From
\begin{gather*}
 \partial_{z_{2}}^{n}\left[w_{\alpha}(z_{1}z_{2})(z_{1}z_{2})^{n}\right] = \partial_{z_{2}}^{n}\left(w_{\beta}(z_{1}z_{2})(z_{1}z_{2})^{m}\right)
 = z_{1}^{m}\partial_{z_{2}}^{n}\left(w_{\beta}(z_{1}z_{2})z_{2}^{m}\right)\\
\hphantom{\partial_{z_{2}}^{n}\left[w_{\alpha}(z_{1}z_{2})(z_{1}z_{2})^{n}\right]}{}
 = z_{1}^{n}w_{\alpha}(z_{1}z_{2})\phi_{n}(z_{1}z_{2};\alpha)
 = z_{1}^{n}z_{2}^{m-n}w_{\beta}(z_{1}z_{2})f_{m,n}(z_{1},z_{2};\beta)
\end{gather*}
 and
\begin{gather*}
\partial_{z_{2}}\left[\left(\frac{z_{2}}{z_{1}}\right)^{m+1-n}w_{\beta}(z_{1}z_{2})f_{m+1,n}(z_{1},z_{2};\beta)\right] \\
 \qquad {}= \partial_{z_{2}}^{n+1}\left(w_{\beta}(z_{1}z_{2})z_{2}^{m+1}\right)
 = \left(\frac{z_{2}}{z_{1}}\right)^{m-n}w_{\beta}(z_{1}z_{2})f_{m+1,n+1}(z_{1},z_{2};\beta)
\end{gather*}
to get \eqref{eq:f-m-n-roderigues-2} and \eqref{eq:f-m-n-shift-n-down-2},
\eqref{eq:f-m-n-roderigues-1} and \eqref{eq:f-m-n-shift-n-down-1}
are proved similarly. Equation \eqref{eq:f-m-n-pde-from-n} is obtained
from~\eqref{eq:f-m-n-shift-n-down-2} and~\eqref{eq:f-m-n-shift-n-down-1}.

From
\begin{gather*}
w_{\beta}(z_{1}z_{2})f_{m,n}(z_{1},z_{2};\beta)=\partial_{z_{1}}^{n}\big(w_{\beta}(z_{1}z_{2})z_{1}^{m}\big)
\end{gather*}
 we get
\begin{gather*}
\partial_{z_{2}} (w_{\beta}(z_{1}z_{2})f_{m,n}(z_{1},z_{2};\beta) )=\partial_{z_{1}}^{n}\big(w_{\beta}^{(1)}(z_{1}z_{2})z_{1}^{m+1}\big).
\end{gather*}
On the other hand
\begin{gather*}
\sum_{j=0}^{t}b_{j}z_{1}^{j+m}z_{2}^{j}w{}_{\beta}^{(1)}(z_{1}z_{2})=\sum_{j=0}^{s}a_{j}z_{1}^{j+m}z_{2}^{j}w_{\beta}(z_{1}z_{2})
\end{gather*}
 implies
\begin{gather*}
\sum_{j=0}^{t}b_{j}z_{2}^{j}\partial_{z_{1}}^{n}\big(z_{1}^{j+m}w{}_{\beta}^{(1)}(z_{1}z_{2})\big)
=\sum_{j=0}^{s}a_{j}z_{2}^{j}\partial_{z_{1}}^{n}\big(z_{1}^{j+m}w_{\beta}(z_{1}z_{2})\big),
\end{gather*}
which gives \eqref{eq:f-m-n-differential-m-recurrence-1}. Equation~\eqref{eq:f-m-n-differential-m-recurrence-2} is proved similarly.
\eqref{eq:f-m-n-pde-from-m} is obtained from~\eqref{eq:f-m-n-differential-m-recurrence-1}
\linebreak and \eqref{eq:f-m-n-differential-m-recurrence-2}.
\end{proof}

The following theorem can be proved similarly.
\begin{Theorem}
\label{theorem:f-m-n-roderigues-q} Assume that
\begin{gather*}
w_{\alpha+\beta}(x)=w_{\alpha+\beta}(x|q)=x^{\alpha}w_{\beta}(x|q)
\end{gather*}
and
\begin{gather*}
\phi_{n}(x;\alpha)=\phi_{n}(x;\alpha|q)=\frac{1}{w_{\alpha}(x)}D_{q,x}^{n}\left(w_{\alpha}(x|q)x^{n}\right),\qquad n=0,1,\dots.
\end{gather*}
 Let $f_{m,n}(z_{1},z_{2};\beta|q)=f_{m,n}(z_{1},z_{2};\beta)$,
then for $m\ge n$ we have
\begin{gather*}
f_{m,n}(z_{1},z_{2};\beta|q)=\frac{\left(\frac{z_{1}}{z_{2}}\right)^{m-n}
\partial_{q,z_{2}}^{n}\left(w_{\beta}(z_{1}z_{2}|q)z_{2}^{m}\right)}{w_{\beta}(z_{1}z_{2}|q)},
\\
f_{m+1,n+1}(z_{1},z_{2};\beta|q)=\frac{\partial_{q,z_{2}}\left(\left(\frac{z_{2}}{z_{1}}\right)^{m+1-n}
w_{\beta}(z_{1}z_{2}|q)f_{m+1,n}(z_{1},z_{2};\beta|q)\right)}{\left(\frac{z_{2}}{z_{1}}\right)^{m-n}w_{\beta}(z_{1}z_{2}|q)},
\\
f_{m,n}(z_{1},z_{2};\beta|q)=\frac{\partial_{q,z_{1}}^{n}\left(w_{\beta}(z_{1}z_{2}|q)z_{1}^{m}\right)}{w_{\beta}(z_{1}z_{2}|q)},
\\
f_{m+1,n+1}(z_{1},z_{2};\beta|q)=\frac{\partial_{q,z_{1}}\left(w_{\beta}(z_{1}z_{2}|q)f_{m+1,n}(z_{1},z_{2};\beta|q)\right)}{w_{\beta}(z_{1}z_{2}|q)}
\end{gather*}
and
\begin{gather*}
 \partial_{q,z_{2}}\left(\left(\frac{z_{2}}{z_{1}}\right)^{m+1-n}w_{\beta}(z_{1}z_{2}|q)f_{m+1,n}(z_{1},z_{2};\beta|q)\right)\\
\qquad{}= \left(\frac{z_{2}}{z_{1}}\right)^{m-n}\partial_{q,z_{1}}\left(w_{\beta}(z_{1}z_{2}|q)f_{m+1,n}(z_{1},z_{2};\beta|q)\right).
\end{gather*}
Furthermore, if for some nonnegative integers $s$, $t$ and real numbers
$\{a_{j}\}_{j=0}^{s}$ and $\{b_{j}\}_{j=0}^{t}$ we have
\begin{gather*}
\frac{D_{q,x}w_{\beta}(x|q)}{w_{\beta}(x|q)}=\frac{\sum\limits_{j=0}^{s}a_{j}x^{j}}{\sum\limits_{j=0}^{t}b_{j}x^{j}},
\end{gather*}
 then for $m\ge n$
\begin{gather*}
\sum_{j=0}^{t}b_{j}\frac{\partial_{q,z_{1}}\big(z_{1}^{n-m-j+1}w_{\beta}(z_{1}z_{2}|q)f_{m+j-1,n}(z_{1},z_{2};\beta|q)\big)}
{z_{1}^{n-m-j}z_{2}^{1-j}w_{\beta}(z_{1}z_{2}|q)}
= \sum_{j=0}^{s}a_{j}z_{2}^{j}f_{m+j,n}(z_{1},z_{2};\beta|q),
\\ \sum_{j=0}^{t}b_{j}z_{2}^{j}\frac{\partial_{q,z_{2}}\big(w_{\beta}(z_{1}z_{2}|q)f_{m+j-1,n}(z_{1},z_{2};\beta|q)\big)}{w_{\beta}(z_{1}z_{2}|q)}
= \sum_{j=0}^{s}a_{j}z_{2}^{j}f_{m+j,n}(z_{1},z_{2};\beta|q)
\end{gather*}
and
\begin{gather*}
 \sum_{j=0}^{t}b_{j}\frac{\partial_{q,z_{1}}\big(z_{1}^{n-m-j+1}w_{\beta}(z_{1}z_{2}|q)f_{m+j-1,n}(z_{1},z_{2};\beta|q)\big)}
 {z_{1}^{n-m-j}z_{2}^{1-j}w_{\beta}(z_{1}z_{2}|q)} \\
 \qquad{}
= \sum_{j=0}^{t}b_{j}z_{2}^{j}\frac{\partial_{q,z_{2}}\big(w_{\beta}(z_{1}z_{2}|q)f_{m+j-1,n}(z_{1},z_{2};\beta|q)\big)}{w_{\beta}(z_{1}z_{2}|q)}.
\end{gather*}
\end{Theorem}

 We now discuss monotonicity of zeros of $\{f_{m,n}(z, \bar z)\}$.
The following version of Markov's theorem is Theorem~7.1.1 in~\cite{Ism}.
\begin{Theorem}\label{theorem7.1.1}
Let $\left\{p_n(x;\tau)\right\}$ be orthogonal with respect to $d\alpha(x;\tau)$,
\begin{gather*}
d\mu(x;\tau)=\rho(x;\tau) d\beta(x),
\end{gather*}
on an interval $I=(a,b)$ and assume that $\rho(x;\tau)$ is positive and has a continuous first derivative with respect to~$\tau$ for~$x\in I$, $\tau\in T=(\tau_1,\tau_2)$. Furthermore, assume that
\begin{gather*}
\int_a^b x^j \rho_\tau(x;\tau) d\beta(x),\qquad j=0,1,\dots,2n-1,
\end{gather*}
converge uniformly for $\tau$ in every compact subinterval of $T$. Then the zeros of the orthogonal polynomials $p_n(x;\tau)$ are increasing
$($decreasing$)$ functions of $\tau$, $\tau\in T$, if $\partial\{\ln\rho(x;\tau)\}/\partial\tau$ is an increasing
$($decreasing$)$ function of~$x$, $x\in I$.
\end{Theorem}
The standard version of Markov's theorem is when $\beta(x) =x$, see~\cite{Sze}.

By applying Theorem~\ref{theorem7.1.1} to the polynomials $\{\phi_n(r; \alpha)\}$ def\/ined in~\eqref{eqphin} it follows that the zeros of $\phi_n(r;\alpha)$ increase with $\alpha, \alpha \ge 0$. Now f\/ix $n$ and assume that $m \ge n$. In view of the def\/inition of the polynomials $\{f_{m,n}(z, \bar z)\}$, $m \ge n$, in~\eqref{eqdeffmn} we see that if~$r_0$ is a~zero of $\phi_n(r; \alpha)$ then $f_{m,n}(z, \bar z) =0$
for all $z$ on the circle $|z| = \sqrt{r_0}$. Since $\alpha = m-n$ the following theorem holds.

\begin{Theorem}
For fixed $n$ the radii of the circles forming the zero sets of $f_{m,n}(z, \bar z)$ increase with $m$ for all $m \ge n$.
\end{Theorem}

 \section[The polynomials $\big\{Z^{(\beta)}_{m,n}(z_1,z_2)\big\}$]{The polynomials $\boldsymbol{\big\{Z^{(\beta)}_{m,n}(z_1,z_2)\big\}}$}\label{section3}

 Motivated by the class of general 1$D$ systems in Section~\ref{section2}
 we def\/ine polynomials
 $\{Z^{(\beta)}_{m, n}(z_1,z_2)\}$ by
\begin{gather}
\label{defZmn}
Z^{(\beta)}_{m, n}(z_1, z_2)
 =
 \frac{1}{n!}
 \sum_{k=0}^{n} \binom{n}{k} \frac{(\beta+1)_m}{(\beta+1)_{m-k}} (-1)^{n-k} z_1^{m-k}
 z_2^{n-k}
\end{gather}
for $m \ge n$. When $ m< n$ the polynomials are def\/ined by
\begin{gather*}
Z^{(\beta)}_{m, n}(z_1, z_2) = Z^{(\beta)}_{n,m}(z_2, z_1)
\end{gather*}
Here $\beta >-1$.

These polynomials arise through the choice $\phi_n(x; \alpha) = L_n^{(\alpha+\beta)}(x)$, where $L_n^{(a)}(x)$ is a Laguerre polynomial \cite{Ism, Sze}
 \begin{gather*}
 L_n^{(a)}(x) = \frac{(a+1)_n}{n!} \sum_{k=0}^n \binom{n}{k} \frac{(-x)^k}{(a+1)_k}=\frac{(a+1)_n}{n!} \sum_{k=0}^n \binom{n}{k}\frac{(-x)^{n-k}}{(a+1)_{n-k}}.
 \end{gather*}
 Indeed, for $m\ge n$
 \begin{gather*}
 Z^{(\beta)}_{m, n}(z_1, z_2) = z_1^{m-n} L_n^{(\beta+ m-n)}(z_1z_2)
 \end{gather*}
 and
 \begin{gather*}
 z_1Z^{(\beta+1)}_{m, n}(z_1, z_2) = Z^{(\beta)}_{m+1, n}(z_1, z_2).
 \end{gather*}
It is clear that when $\beta =0$ we see that
\begin{gather*}
H_{m, n}(z_1, z_2) = (-1)^n(m\wedge n)! Z^{(0)}_{m, n}(z_1, z_2).
\end{gather*}

Therefore in the notation of Section~\ref{section2}, we have
\begin{gather}
\label{eqzetabdcjn}
\zeta_n(\alpha) = \frac{\Gamma(\alpha+\beta+n+1)}{n!},
\qquad c_j(n,\alpha) = \frac{(-1)^{n-j} (\alpha+\beta+1)_n}{(n-j)! j! (\alpha+\beta+1)_{n-j}}.
\end{gather}

\begin{Theorem}\label{ortZmn}
For $m,n,s,t=0,1,\dots$ and $\beta>-1$, we have the orthogonality relation
\begin{gather*}
\int_{\mathbb{R}^2} Z^{(\beta)}_{m, n}(z, \bar z) \overline{Z^{(\beta)}_{s,t}(z, \bar z)} (z\bar z)^\beta e^{-z \bar z} r dr d\theta = \pi \frac{\Gamma(\beta+m\vee n+1)}{(m\wedge n)!}
\delta_{m,s}\delta_{n,t} .
\end{gather*}
\end{Theorem}
This follows from Theorem \ref{orthfmn} in the case~\eqref{eqzetabdcjn} but for completeness we give a direct proof.

\begin{proof}[Alternate proof]
When $m \ge n$ we consider the integral
\begin{gather*}
I(m,n,a,b) := \int_{\mathbb{R}^2} Z^{(\beta)}_{m, n}(z, \bar z) \overline{(z^a \bar z^b)} r^{2\beta+1} e^{-r^2} dr d \theta,
\end{gather*}
for $a \le m, b \le n$.
Therefore
\begin{gather*}
I(m,n,a,b) := 2\pi \delta_{b+m, a+n} \frac{(-1)^n}{n!} \sum_{k=0}^{n} \binom{n}{k} \frac{(\beta+1)_m}{(\beta+1)_{m-k}} (-1)^k
\int_0^\infty r^{a+ b+m+n -2k+ 2\beta+1} e^{-r^2} dr \\
\hphantom{I(m,n,a,b)}{}
= \pi \delta_{b+m, a+n} \frac{(\beta+1)_m}{n!} \sum_{k=0}^{ n} \binom{n}{k} \frac{(-1)^{n-k}}{(\beta+1)_{m-k}} \Gamma(1+ \beta-k+ (a+b+m+n)/2) \\
\hphantom{I(m,n,a,b)}{}
= \pi \delta_{b+m, a+n} \frac{(\beta+1)_m}{n!} \sum_{k=0}^ n \binom{n}{k} \frac{(-1)^{n-k}}{(\beta+1)_{m-k}} \Gamma(1+ \beta-k+ a+n) \\
\hphantom{I(m,n,a,b)}{}
= \pi \delta_{b+m, a+n} \frac{ \Gamma(1+ \beta + a)}{(\beta+1)_{m-n}} \frac{(\beta+1)_m}{n!}
\, {}_2F_1(-n, 1+a +\beta; \beta+1+m-n;1) \\
\hphantom{I(m,n,a,b)}{}
= \pi \delta_{b+m, a+n} \frac{ \Gamma(1+ \beta + a)}{(\beta+1)_{m-n}}
 \frac{(\beta+1)_m}{n!} \frac{(m-n-a)_n}{(\beta+1+m-n)_n}.
\end{gather*}
Since $m-n-a = -b$ we that the integral vanishes for $b <n$. If $b=n$ then $a =m$ and we conclude that
\begin{gather*}
 \int_{\mathbb{R}^2} Z^{(\beta)}_{m, n}(z ,\bar z) \overline{(z^a \bar z^b)} r^{2\beta+1} e^{-r^2} dr d \theta
 = \pi \Gamma(1+\beta+m) \delta_{a, m} \delta_{b, n},
\end{gather*}
and the theorem now follows.
\end{proof}

In view of \eqref{eqzetabdcjn} the recurrence relations \eqref{eqrecrel2}, \eqref{eq3trrG2D} and \eqref{eqrecrel2-1-2} become
\begin{gather}
\label{eqrrZmn1}
z_1Z^{(\beta)}_{m, n}(z_1, z_2) = Z^{(\beta)}_{m+1, n}(z_1, z_2) - Z^{(\beta)}_{m, n-1}(z_1, z_2), \\
\label{eqrrZmn2}
z_2 Z^{(\beta)}_{m+1, n}(z_1, z_2) = -(n+1) Z^{(\beta)}_{m+1, n+1}(z_1, z_2) +
(\beta+m+1)Z^{(\beta)}_{m, n}(z_1, z_2),
\end{gather}
and
\begin{gather*}
 (\beta+m+n+1-z_{1}z_{2} )Z_{m,n}^{(\beta)}(z_{1},z_{2})\\
\qquad{} = (n+1 )Z_{m+1,n+1}^{(\beta)}(z_{1},z_{2})+ (m+\beta )Z_{m-1,n-1}^{(\beta)}(z_{1},z_{2})
\end{gather*}
respectively, where $m\ge n$.
We now discuss dif\/ferential recurrence relations.
 \begin{Theorem}
 For $m\ge n$, the polynomials $\{Z^{(\beta)}_{m, n}(z_1, z_2) \}$ satisfy the differential recurrence relations
 \begin{gather}
 {\delta}_{ z_2} Z^{(\beta)}_{m, n}(z_1, z_2) = -z_2 Z^{(\beta)}_{m, n-1}(z_1, z_2),
 \label{eqdrr1}
 \\
 \delta _{z_{2}}Z_{m,n}^{(\beta)}(z_{1},z_{2}) = nZ_{m,n}^{(\beta)}(z_{1},z_{2})-(\beta+m)Z_{m-1,n-1}^{(\beta)}(z_{1},z_{2}),
 \label{eqdrr1a}
 \\
 {\delta}_ {z_1} Z^{(\beta)}_{m, n}(z_1, z_2) = (m-n) Z^{(\beta)}_{m, n}(z_1, z_2) - {z_2} Z^{(\beta)}_{m, n-1}(z_1, z_2),
 \label{eqdrr2}
 \\
\label{eqdrr2a}
{\delta}_{z_{1}}Z_{m,n}^{(\beta)}(z_{1},z_{2}) = m Z_{m,n}^{(\beta)}(z_{1},z_{2})-(m+\beta)Z_{m-1,n-1}^{(\beta)}(z_{1},z_{2}),
\\
\label{symmetric-first-order-pde-a}
({\delta}_{z_1}-{\delta}_{z_2})Z_{m,n}^{(\beta)}(z_1,z_2)=(m-n)Z_{m,n}^{(\beta)}(z_1,z_2),
\\
\label{symmetric-first-order-pde-b}
(n{\delta}_{z_1}-m{\delta}_{z_2})Z_{m,n}^{(\beta)}(z_1,z_2)=(m-n)(m+\beta)Z_{m-1,n-1}^{(\beta)}(z_1,z_2),
 \end{gather}
where ${\delta}_{z_1}=z_1{\partial}_{z_1}$, ${\delta}_{z_2}=z_2{\partial}_{z_2}$.
 Moreover they have the operational representation
 \begin{gather}
 Z^{(\beta)}_{m, n}(z_1, z_2) = \frac{(-1)^n}{ n!} z_1^{-\beta} \exp(-\partial_{z_1} \partial_{z_2}) z_1^{\beta+m}z_2^n
 \label{eqexpdd}
 \end{gather}
 and the Rodrigues type representation
 \begin{gather}
 Z^{(\beta)}_{m, n}(z_1, z_2) = \frac{1}{n!} z_1^{-\beta}e^{z_1z_2}
 \partial_{z_1}^n \big(z_1^{m+\beta}e^{-z_1z_2}\big) \nonumber\\
 \hphantom{Z^{(\beta)}_{m, n}(z_1, z_2)}{} =
 \frac{(-1)^{m}}{n!} (z_1z_2)^{-\beta}e^{z_1z_2} \partial_{z_1}^n
 \big((z_1z_2)^\beta \partial_{z_2}^m \big(e^{-z_1z_2}\big)\big).
 \label{eqRoZmn}
 \end{gather}
 \end{Theorem}
 \begin{proof}
 Observe that from \eqref{defZmn} we obtain
 \begin{gather*}
\partial_{z_{2}}Z_{m,n}^{(\beta)}(z_{1},z_{2})=\sum_{k=0}^{n}\frac{(-1)^{n-k}}{k!(n-k)!}\frac{(\beta+1)_{m}}{(\beta+1)_{m-k}}z_{1}^{m-k}(n-k)z_{2}^{n-k-1}.
\end{gather*}
 Therefore we have
\begin{gather*}
\partial_{z_{2}}Z_{m,n}^{(\beta)}(z_{1},z_{2}) = -\sum_{k=0}^{n-1}\frac{(-1)^{n-1-k}}{k!(n-1-k)!}\frac{(\beta+1)_{m}}{(\beta+1)_{m-k}}z_{1}^{m-k}z_{2}^{n-k-1}
 = -Z_{m,n-1}^{(\beta)}(z_{1},z_{2})
\notag
\end{gather*}
 and
\begin{gather*}
\partial_{z_{2}}Z_{m,n}^{(\beta)}(z_{1},z_{2}) = \frac{n}{z_{2}}\sum_{k=0}^{n}\frac{(-1)^{n-k}}{k!(n-k)!}\frac{(\beta+1)_{m}}{(\beta+1)_{m-k}}z_{1}^{m-k}z_{2}^{n-k}\\
\hphantom{\partial_{z_{2}}Z_{m,n}^{(\beta)}(z_{1},z_{2}) = }{}
 - \sum_{k=1}^{n}\frac{(-1)^{n-k}}{(k-1)!(n-k)!}\frac{(\beta+1)_{m}}{(\beta+1)_{m-k}}z_{1}^{m-k}z_{2}^{n-k-1}\\
\hphantom{\partial_{z_{2}}Z_{m,n}^{(\beta)}(z_{1},z_{2}) }{} = \frac{n}{z_{2}}Z_{m,n}^{(\beta)}(z_{1},z_{2})+\sum_{j=0}^{n-1}\frac{(-1)^{n-j}}{j!(n-1-j)!}\frac{(\beta+1)_{m}z_{1}^{m-1-j}z_{2}^{n-j-2}}{(\beta+1)_{m-1-j}}\\
\hphantom{\partial_{z_{2}}Z_{m,n}^{(\beta)}(z_{1},z_{2}) }{}
 = \frac{n}{z_{2}}Z_{m,n}^{(\beta)}(z_{1},z_{2})-\frac{(\beta+m)}{z_{2}}
 \sum_{j=0}^{n-1}\frac{(-1)^{n-1-j}}{j!(n-1-j)!}\frac{(\beta+1)_{m-1}z_{1}^{m-1-j}z_{2}^{n-1-j}}{(\beta+1)_{m-1-j}}\\
\hphantom{\partial_{z_{2}}Z_{m,n}^{(\beta)}(z_{1},z_{2}) }{}
 = \frac{n}{z_{2}}Z_{m,n}^{(\beta)}(z_{1},z_{2})-\frac{(\beta+m)}{z_{2}}Z_{m-1,n-1}^{(\beta)}(z_{1},z_{2}).
 \end{gather*}
The relationships \eqref{eqdrr2} and \eqref{eqdrr2a} are proved similarly. Formulas~\eqref{eqexpdd} and~\eqref{eqRoZmn} follow by direct calculation using \eqref{defZmn}. The f\/irst-order partial dif\/ferential equation~\eqref{symmetric-first-order-pde-a} can be obtained from either \eqref{eqdrr1} and \eqref{eqdrr2} or \eqref{eqdrr1a} and~\eqref{eqdrr2a}, \eqref{symmetric-first-order-pde-b} is proved similarly.
\end{proof}

 One would have expected the Rodrigues formula for $Z^{(\beta)}_{m, n}(z_1, z_2)$ to be a constant multiple of
 \begin{gather*}
 (z_1z_2)^{-\beta}e^{z_1z_2} \partial_{z_1}^n \partial_{z_2}^m
 \big((z_1z_2)^\beta e^{-z_1z_2}\big),
 \end{gather*}
which is obviously false except in the 2$D$-Hermite case $\beta =0$.

 Note that \eqref{eqrrZmn2} follows from \eqref{eqRoZmn} by writing $\partial_{z_1}^n $ as
 $\partial_{z_1}^{n-1} \partial_{z_1}$.

 We shall use the notation
 \begin{gather*}
 w_\beta(z_1z_2) = (z_1z_2)^\beta \exp(-z_1z_2).
 \end{gather*}
It is clear that \eqref{eqRoZmn} implies the dif\/ferentiation formulas
\begin{gather*}
e^{z_{1}z_{2}}\partial_{z_{2}}\big(e^{-z_{1}z_{2}}Z_{m,n}^{(\beta)}(z_{1},z_{2})\big)
=-z_{1}Z_{m,n}^{(\beta+1)}(z_{1},z_{2})=-Z_{m+1,n}^{(\beta)}(z_{1},z_{2})
\end{gather*}
and
\begin{gather}
\label{shift-n-index-up-by-1}
 \frac{1}{w_\beta(z_1 z_2)} {\partial}_{z_1} \big(w_\beta(z_1z_2) Z^{(\beta)}_{m, n}(z_1, z_2)\big) =(n+1) Z^{(\beta)}_{m, n+1}(z_1, z_2)
\end{gather}
for $m\ge n$.
It also clear that
\begin{gather}
\label{weight-derivatives}
{\delta}_{z_j} w_\beta(z_1z_2) = ( \beta - z_1z_2) w_\beta(z_1z_2),
\qquad j=1,2
 \end{gather}
and \eqref{eqrrZmn1} which imply the relationships
\begin{gather*}
\frac{{\delta}_{z_2}\big(w_\beta(z_1z_2) Z^{(\beta)}_{m, n}(z_1, z_2)\big)}{w_\beta(z_1 z_2)}=\beta Z^{(\beta)}_{m, n}(z_1, z_2)-z_2Z^{(\beta)}_{m+1, n}(z_1, z_2),
\end{gather*}
and
\begin{gather*}
\frac{{\delta}_{z_1}\big(w_\beta(z_1z_2) Z^{(\beta)}_{m, n}(z_1, z_2)\big)}{w_\beta(z_1 z_2)}=(\beta+m-n) Z^{(\beta)}_{m, n}(z_1, z_2)-z_2Z^{(\beta)}_{m+1, n}(z_1, z_2).
\end{gather*}
This can be seen by f\/irst applying \eqref{weight-derivatives}, then applying \eqref{eqrrZmn1}, for example,
\begin{gather*}
 \frac{1}{w_\beta(z_1 z_2)} {\partial}_ {z_2} \big(w_\beta(z_1z_2) Z^{(\beta)}_{m, n}(z_1, z_2)\big) \\
\qquad{} = \frac{\beta -z_1z_2}{z_2} Z^{(\beta)}_{m, n}(z_1, z_2)-Z^{(\beta)}_{m, n-1}(z_1, z_2)
 = \frac{\beta}{z_2} Z^{(\beta)}_{m, n}(z_1, z_2) - Z^{(\beta)}_{m+1, n}(z_1, z_2).
\end{gather*}

\begin{Theorem}
 The polynomials $\{Z^{(\beta)}_{m, n}(z_1, z_2)\}$ satisfy the second-order partial
 differential equation
 \begin{gather}
 \label{eqPDE1}
 \partial_{\partial z_1} \partial_{\partial z_2} f +
 \left( \frac{\beta-z_1z_2}{z_1} \right)
 \partial_{\partial z_2} f = -n f,
 \end{gather}
 for all $m \ge n$.
 \end{Theorem}
 \begin{proof}
 From equations \eqref{shift-n-index-up-by-1}, \eqref{weight-derivatives} and \eqref{eqdrr1} we get
\begin{gather*}
\frac{\partial_{z_{1}}w_{\beta}(z_{1}z_{2})Z_{m,n}^{(\beta)}(z_{1},z_{2})+w_{\beta}(z_{1}z_{2})\partial_{z_{1}}Z_{m,n}^{(\beta)}(z_{1},z_{2})}{w_{\beta}(z_{1}z_{2})}\\
\qquad{} = \left(\frac{\beta}{z_{1}}-z_{2}\right)Z_{m,n}^{(\beta)}(z_{1},z_{2})+\partial_{z_{1}}Z_{m,n}^{(\beta)}(z_{1},z_{2})=(n+1)Z_{m,n+1}^{(\beta)}(z_{1},z_{2})
\end{gather*}
 and
\begin{gather*}
 \partial_{z_{2}}\left\{ \left(\frac{\beta}{z_{1}}-z_{2}\right)Z_{m,n}^{(\beta)}(z_{1},z_{2})\right\} +\partial_{z_{1}}\partial_{z_{2}}Z_{m,n}^{(\beta)}(z_{1},z_{2})\\
\qquad{} = (n+1)\partial_{z_{2}}Z_{m,n+1}^{(\beta)}(z_{1},z_{2})=-(n+1)Z_{m,n}^{(\beta)}(z_{1},z_{2}),
\end{gather*}
 which simplif\/ies to the desired second-order partial dif\/ferential equation.
 \end{proof}

 It is clear that the dif\/ferential property \eqref{eqPDE1} can be written in the form
 \begin{gather}
 \label{eqPDE2}
 \frac{ \partial}{\partial z_1} \left[ w_\beta(z_1z_2) \frac{\partial}{\partial z_2}Z^{(\beta)}_{m, n}(z_1, z_2)\right] = -n Z^{(\beta)}_{m, n}(z_1, z_2).
 \end{gather}

 Note that \eqref{eqPDE1} and \eqref{symmetric-first-order-pde-a} indicate that the polynomials
 $\{Z_{m,n}^{(\beta)}(z_{1},z_{2})\}$ are simultaneous eigenfunctions of the operators on their respective left-hand sides. The dif\/ferential operators
 $\partial_{\partial z_1} \partial_{\partial z_2} +
 \left( \frac{\beta-z_1z_2}{z_1} \right)
 \partial_{\partial z_2}$ and $z_1 \partial_{\partial z_1} - z_2 \partial_{\partial z_2}$
 indeed commute as can be directly verif\/ied.

 Let $w(x,y) = (z\bar z)^\beta e^{-z\bar z}=r^{2\beta+1}e^{-r^2}dr d\theta$. It is clear that the bilinear functional
\begin{gather*}
\langle f, g\rangle = \int_{\mathbb{R}^2} f(x,y) \overline{g(x,y)} w(x,y) dx dy
\end{gather*}
def\/ines a semi-inner product on the linear space of $w(x,y)dx dy$ measurable functions $ f(x,y) $ such that $\int_{\mathbb{R}^2} |f(x,y)|^2 w(x,y) dx dy<\infty$. Since a polynomial $p(x,y)$ is also a polynomial in $r$ with trigonometric polynomials as coef\/f\/icients, hence it is a continuous function in $(r,\theta)$. Then $\int_{\mathbb{R}^2} |p(x,y)|^2 w(x,y) dx dy=0$ implies that $p(x,y)=0$, consequently the bilinear functional is positive def\/inite on the subspace of all the polynomials~$p(x,y)$. In the subsequent discussion we only need this inner product space of all the polynomials.
 \begin{Theorem}
 We have the adjoint relations
 \begin{gather*}
 (\partial_z)^* = - \partial_{\bar z} -\frac{\beta}{\bar z} + z, \qquad
 (\partial_{\bar z})^* = - \partial_{ z} -\frac{\beta}{z} + \bar z.
 \end{gather*}
 \end{Theorem}
 The proof is straightforward calculus exercise.

 This allows us to write \eqref{eqPDE1} or \eqref{eqPDE2} in the selfadjoint form
 \begin{gather*}
\left( (\partial_{\bar z})^* \partial_{\bar z}\right) f = n f.
 \end{gather*}
 Moreover we also have the following result.

 \begin{Theorem}
 The operator $A := (\partial_{\bar z})^* \partial_{\bar z}$ is positive in the sense
 that $\langle Af,f\rangle \ge 0$ with $=$ if and only if $\partial_{\bar z} f =0$, that is $f$ depends only on~$z$.
 \end{Theorem}

 This indicates that
 \begin{gather*}
 \frac{1}{w(x,y)} \frac{\partial}{\partial z} Z_{m,n}^{(\beta)}(z, \bar z) = (n+1)
 Z_{m,n}^{(\beta)}(z, \bar z).
 \end{gather*}
 This is the adjoint of \eqref{eqdrr1}.

 Ismail and Zeng \cite{Ism:Zen1} established the connection relations stated in the next theorem.
 \begin{Theorem}\label{theorem3.4}
We have the connection relation
\begin{gather*}
Z_{m,n}^{(\beta)}(z_1,z_2)=\sum_{j=0}^n \frac{(\beta-\gamma)_j}{j!}(-1)^j
Z_{m-j,n-j}^{(\gamma)}(z_1,z_2), \qquad m\geq n.
\end{gather*}
 For $m \ge n$, we have the special cases
\begin{gather*}
n! Z_{m,n}^{(\beta)}(z_1,z_2) = \sum_{j=0}^n {n\choose j} (\beta)_j(-1)^j
H_{m-j,n-j}(z_1,z_2), \\
H_{m,n}(z_1,z_2) = n! \sum_{j=0}^n \frac{(-\beta )_j}{j!}(-1)^j
Z_{m-j,n-j}^{(\beta)}(z_1,z_2).
\end{gather*}
\end{Theorem}

The Laguerre dif\/ferential equation is \cite{Rai,Sze}
\begin{gather*}
x y^{\prime\prime} + (1+\alpha-x)y^\prime + ny =0.
\end{gather*}
Therefore Theorem \ref{theorem2.4} shows that
\begin{gather*}
z_2 \frac{\partial^2 Z_{m,n}^{(\beta)}}{\partial z_2^2} + (1+\beta + m-n -z_1z_2)
\frac{\partial Z_{m,n}^{(\beta)}}{\partial z_2} + n z_1 Z_{m,n}^{(\beta)} = 0.
\end{gather*}

Ismail and Zeng \cite{Ism:Zen1} gave the generating function
\begin{gather}
\sum_{m \ge n\geq 0} \frac{u^mv^n}{(m-n)!} Z_{m,n}^{(\beta)}(z_1, z_2)
= (1+ uv )^{-\beta-1} \exp\left( \frac{uvz_1z_2+z_1u}{1+uv}\right),
\label{eqgfZ2}
\end{gather}
this can be seen from
\begin{gather*}
g_{\alpha}(x,z)=\sum_{n=0}^{\infty}L_{n}^{(\alpha+\beta)}(x)z^{n}=\frac{\exp\big\{ {-}\frac{xz}{1-z}\big\} }{(1-z)^{\alpha+\beta+1}},
\end{gather*}
and \eqref{eq:f-m-n-generating-function},
\begin{gather*}
\sum_{m\ge n\ge0}\frac{u^{m}v^{n}}{(m-n)!}Z_{m,n}^{(\beta)}(z_{1},z_{2}) = \frac{\exp\big\{ \frac{uz_{1}-z_{1}z_{2}uv}{1-uv}\big\} }{\left(1-uv\right)^{\beta+1}},\\
\sum_{m\ge n\ge0}u^{m}v^{n}Z_{m,n}^{(\beta)}(z_{1},z_{2}) = \frac{\exp\big\{ {-}\frac{z_{1}z_{2}uv}{1-uv}\big\} }{(1-uv)^{\beta} (1-z_{1}u-uv )}.
\end{gather*}
It is clear that the generating function \eqref{eqgfZ2} implies the identity
\begin{gather*}
 Z_{m,n}^{(\beta+ \gamma+1)}(z_1+ z_3, z_2+z_4) = \sum_{m\ge j \ge k \ge 0}
 \frac{Z_{j,k}^{(\beta)}(z_1, z_2)Z_{m-j,n-k}^{( \gamma)}( z_3, z_4)}{(j-k)! (m-n -j +k)!}.
\end{gather*}
This is an analogue of the convolution identity
\begin{gather*}
L_n^{(\alpha+\beta+1)}(x+y) = \sum_{k=0}^n L_k^{(\alpha)}(x)L_n^{(\beta)}(y).
\end{gather*}
Al-Salam and Chihara characterized all 1-$D$ orthogonal polynomials satisfying convolution formulas in~\cite{Als:Chi}. They discovered the Al-Salam--Chihara polynomials through this characterization~\cite{Ism}. It will be interesting to solve the corresponding 2-$D$ characterization problem.

\section[The polynomials $\big\{M_n^{(\beta,\gamma)}(z_1,z_2)\big\}$]{The polynomials $\boldsymbol{\big\{M_n^{(\beta,\gamma)}(z_1,z_2)\big\}}$}\label{section4}

For $\alpha>0,\beta,\gamma>-1$ Ismail and Zeng~\cite{Ism:Zen1} introduced the polynomials $ \phi_n(r,\alpha)= P_n^{(\alpha+\gamma, \beta)}(1-2r)$, that is
\begin{gather}
\label{phi-beta-gamma-definition}
 \phi_n(r,\alpha)=(\alpha+\gamma+1)_n \sum_{k=0}^{n}\frac{(\alpha+\beta+\gamma+n+1)_{n-k}(-r)^{n-k}}{k!(n-k)!(\alpha+\gamma+1)_{n-k}}.
\end{gather}
They satisfy the following orthogonality relation
\begin{gather*}
\int_{0}^{1} \phi_m(r,\alpha) \phi_n(r,\alpha)u^{\alpha+\gamma}(1-u)^\beta du= \zeta_n(\alpha+\gamma,\beta) \delta_{m,n},
\end{gather*}
where
\begin{gather*}
\zeta_n(\alpha+\gamma,\beta) =\frac{\Gamma(\alpha+\gamma+n+1)\Gamma(\beta+n+1)}{n!\Gamma(\alpha+\beta+\gamma+n+1)(\alpha+\beta+\gamma+2n+1)}.
\end{gather*}

We def\/ine the two variable polynomial $M_{m,n}^{(\beta,\gamma)}(z_1,z_2)$ by
\begin{gather*}
M_{m,n}^{(\beta,\gamma)}(z_1,z_2)=\sum_{k=0}^{n}\frac{(m+\beta+\gamma+1)_{n-k}(\gamma+1)_m z_1^{m-k}(-z_2)^{n-k}}{k!(n-k)!(\gamma+1)_{m-k}}
\end{gather*}
for $m\ge n$ and
\begin{gather*}
M_{m,n}^{(\beta,\gamma)}(z_1,z_2)=M_{n,m}^{(\beta,\gamma)}(z_2,z_1),\qquad m<n.
\end{gather*}
Then
\begin{gather*}
z_1 M_{m,n}^{(\beta,\gamma+1)}(z_1,z_2)=M_{m+1,n}^{(\beta,\gamma)}(z_1,z_2).
\end{gather*}
Clearly, $M_{m,n}^{(\beta,\gamma)}(z_1,z_2)$ satisfy the orthogonality relation
\begin{gather*}
\int_{|z|\leq 1} M_{m,n}^{(\beta, \gamma)}(z, \bar z)\overline{ M_{p,q}^{(\beta, \gamma)}(z, \bar z)} r^{2\gamma}
\big(1-r^2\big)^\beta 2r dr
 \frac{d\theta}{2\pi}
 \\
\qquad{}= \frac{\Gamma(\gamma+ m+1) \Gamma(\beta+n+1)}{n!\Gamma(\beta+\gamma+m+1) (\beta +\gamma+m+n+1)}
\delta_{m,p}\delta_{n, q},
\end{gather*}
for $m \ge n$ as it is stated in \cite{Ism:Zen1}. The disk polynomials def\/ined in \cite{Dun:Xu} can be expressed as
\begin{gather*}
P^{\beta}_{m,n}(z)=\frac{(-1)^n n!}{(\beta+1)_n}M_{m,n}^{(\beta,0)}(z,\bar{z}),\qquad m\ge n.
\end{gather*}
From \eqref{phi-beta-gamma-definition} we get
\begin{gather*}
c_k(n,\alpha)=\frac{(m+\beta+\gamma+1)_{n-k}(\gamma+1)_m (-1)^{n-k}}{k!(n-k)!(\gamma+1)_{m-k}},
\end{gather*}
where $\alpha=m-n$ with $m\ge n$.
Then
\begin{gather*}
 (\beta+\gamma+m+n+2 ) z_{2}M_{m+1,n}^{(\beta,\gamma)}(z_{1},z_{2})\\
\qquad{} =(\gamma+m+1)M_{m,n}^{(\beta,\gamma)}(z_{1},z_{2})-(n+1)M_{m+1,n+1}^{(\beta,\gamma)}(z_{1},z_{2}),
\\
 (\beta+\gamma+m+n+1 )z_{1}M_{m,n}^{(\beta,\gamma)}(z_{1},z_{2})\\
\qquad{} = (\beta+\gamma+m+1 )M_{m+1,n}^{(\beta,\gamma)}(z_{1},z_{2})- (\beta+n )M_{m,n-1}^{(\beta,\gamma)}(z_{1},z_{2})
\end{gather*}
and
\begin{gather}
\label{M-m-n-recurrence-3}
(c_{n}-z_{1}z_{2})M_{m,n}^{(\beta,\gamma)}(z_{1},z_{2})=a_{n}M_{m+1,n+1}^{(\beta,\gamma)}(z_{1},z_{2})+b_{n}M_{m-1,n-1}^{(\beta,\gamma)}(z_{1},z_{2}),
\end{gather}
where
\begin{gather*}
a_{n} = \frac{(n+1)(\beta+\gamma+m+1)}{(\beta+\gamma+m+n+1)(\beta+\gamma+m+n+2)},\\
b_{n} = \frac{(\gamma+m)(\beta+n)}{(\beta+\gamma+m+n)(\beta+\gamma+m+n+1)},\\
c_{n} = \frac{(n+1)(\gamma+m+1)}{\beta+\gamma+m+n+2}-\frac{n(\gamma+m)}{\beta+\gamma+m+n}.
\end{gather*}
It must be noted that \eqref{M-m-n-recurrence-3} is essentially the three term recurrence relation for Jacobi polynomials~\cite{Sze}.

It is clear that
\begin{gather*}
A_{n}\delta_{x}^{2}\phi_{n}(x,\alpha)+B_{n}\delta_{x}\phi_{n}(x,\alpha)+C_{n}\phi_{n}(x,\alpha)=0,
\end{gather*}
where
\begin{gather*}
A_{n} = 1-x,\qquad 
B_{n} = \alpha+\gamma-x (\alpha+\beta+\gamma+1 ),\qquad 
C_{n} = xn (\alpha+\beta+\gamma+n+1 ).
\end{gather*}
Then for $m\ge n$, $M_{m,n}^{(\beta,\gamma)}(z_{1},z_{2})$ satisfy
the following second-order partial dif\/ferential equation
\begin{gather*}
 (1-z_{1}z_{2} )\delta_{z_{2}}^{2}f (z_{1},z_{2} )+ \{ m-n+\gamma
- z_{1}z_{2} (\beta+\gamma+m-n+1 ) \} \delta_{z_{2}}f(z_{1},z_{2})\\
\qquad{} + z_{1}z_{2}n (\beta+\gamma+m+1 )f(z_1,z_2) =0.
\end{gather*}
An equivalent form is
\begin{gather*}
(1-z_1z_2) z_2 \frac{\partial^2 f(z_1,z_2)}{\partial z_2^2} +
 [1+m-n+\gamma -z_1z_2(2+\gamma+\beta +m-n) ] \frac{\partial f(z_1,z_2)}{\partial z_2} \\
\qquad {}+ z_1n (m+\beta+\gamma +1) f(z_1,z_2)=0.
\end{gather*}
Similarly we establish
\begin{gather*}
(1-z_1z_2) z_1 \frac{\partial^2 f(z_1,z_1)}{\partial z_2^1} +
 [1-m+n+\gamma -z_1z_2(2+\gamma+\beta -m+n) ] \frac{\partial f(z_1,z_2)}{\partial z_2} \\
\qquad {}+ z_2m (n+\beta+\gamma +1) f(z_1,z_2)=0.
\end{gather*}

\begin{Theorem}
For $m\ge n\ge0$, $\beta> -1$, $\gamma>-1$ we have
\begin{gather*}
 \frac{\partial} {\partial z_2}M_{m,n}^{(\beta,\gamma)}(z_{1},z_{2}) = -(\beta+\gamma+m+1)
 M_{m,n}^{(\beta+1,\gamma)}(z_{1},z_{2}), \\
 \delta_{z_{2}}M_{m,n}^{(\beta,\gamma)}(z_{1},z_{2}) = nM_{m,n}^{(\beta,\gamma)}(z_{1},z_{2})- (\gamma+m )M_{m-1,n-1}^{(\beta+1,\gamma)}(z_{1},z_{2})\\
 \hphantom{\delta_{z_{2}}M_{m,n}^{(\beta,\gamma)}(z_{1},z_{2})}{}
 = (\beta+\gamma+m+1 )\big(M_{m,n}^{(\beta+1,\gamma)}(z_{1},z_{2})-M_{m,n}^{(\beta,\gamma)}(z_{1},z_{2})\big), \\
\delta_{z_{1}}M_{m,n}^{(\beta,\gamma)}(z_{1},z_{2}) =mM_{m,n}^{(\beta,\gamma)}(z_{1},z_{2})- (\gamma+m )M_{m-1,n-1}^{(\beta+1,\gamma)}(z_{1},z_{2})\\
\hphantom{\delta_{z_{1}}M_{m,n}^{(\beta,\gamma)}(z_{1},z_{2})}{}
 = (\beta+\gamma+m+1 )M_{m,n}^{(\beta+1,\gamma)}(z_{1},z_{2})
 - (\beta+\gamma+n+1 )M_{m,n}^{(\beta,\gamma)}(z_{1},z_{2}), \\
 (\beta+\gamma+m+1 )\delta_{z_{1}}M_{m,n}^{(\beta,\gamma)}(z_{1},z_{2})
 - (\beta+\gamma+n+1 )\delta_{z_{2}}M_{m,n}^{(\beta,\gamma)}(z_{1},z_{2})\\
 \qquad{}
 = (m-n ) (\beta+\gamma+m+1 )M_{m,n}^{(\beta+1,\gamma)}(z_{1},z_{2}),
 \\
 (\delta_{z_{1}}-\delta_{z_{2}} )M_{m,n}^{(\beta,\gamma)}(z_{1},z_{2}) = (m-n )M_{m,n}^{(\beta,\gamma)}(z_{1},z_{2}), \\
 (\beta+\gamma+m+1 )\delta_{z_{1}}M_{m,n}^{(\beta,\gamma)}(z_{1},z_{2})
- (\beta+\gamma+n+1 )\delta_{z_{2}}M_{m,n}^{(\beta,\gamma)}(z_{1},z_{2})\\
\qquad{}= (m-n ) (\beta+\gamma+m+1 )M_{m,n}^{(\beta+1,\gamma)}(z_{1},z_{2}).
\end{gather*}
 \end{Theorem}

\begin{proof}The f\/irst equation easily follows from the def\/inition.
From the def\/inition for $\!M_{m{,}n}^{(\beta{,}\gamma)}\!(z_{1}{,}z_{2})\!$ we
get
\begin{gather*}
 \delta_{z_{2}}M_{m,n}^{(\beta,\gamma)}(z_{1},z_{2})
 = \sum_{k=0}^{n}\frac{ (\beta+\gamma+m+1 )_{n-k} (\gamma+1 )_{m}z_{1}^{m-k} (-z_{2} )^{n-k} (n-k )}{k! (n-k )! (\gamma+1 )_{m-k}}\\
\hphantom{\delta_{z_{2}}M_{m,n}^{(\beta,\gamma)}(z_{1},z_{2}) }{}
 = \sum_{k=0}^{n}\frac{(\beta+\gamma+m+1)_{n-k+1}(\gamma+1)_{m}z_{1}^{m-k}(-z_{2})^{n-k}}{k!(n-k)!(\gamma+1)_{m-k}}\\
\hphantom{\delta_{z_{2}}M_{m,n}^{(\beta,\gamma)}(z_{1},z_{2}) =}{}
 - (\beta+\gamma+m+1)\sum_{k=0}^{n}\frac{(\beta+\gamma+m+1)_{n-k}(\gamma+1)_{m}z_{1}^{m-k}(-z_{2})^{n-k}}{k!(n-k)!(\gamma+1)_{m-k}}\\
\hphantom{\delta_{z_{2}}M_{m,n}^{(\beta,\gamma)}(z_{1},z_{2}) }{}
 = (\beta+\gamma+m+1)M_{m,n}^{(\beta+1,\gamma)}(z_{1},z_{2})-(\beta+\gamma+m+1)M_{m,n}^{(\beta,\gamma)}(z_{1},z_{2})
\end{gather*}
and
\begin{gather*}
 \delta_{z_{2}}M_{m,n}^{(\beta,\gamma)}(z_{1},z_{2})
 = \sum_{k=0}^{n}\frac{(\beta+\gamma+m+1)_{n-k}(\gamma+1)_{m}z_{1}^{m-k}(-z_{2})^{n-k}(n-k)}{k!(n-k)!(\gamma+1)_{m-k}}\\
\hphantom{\delta_{z_{2}}M_{m,n}^{(\beta,\gamma)}(z_{1},z_{2}) }{}
 = n\sum_{k=0}^{n}\frac{(\beta+\gamma+m+1)_{n-k}(\gamma+1)_{m}z_{1}^{m-k}(-z_{2})^{n-k}}{k!(n-k)!(\gamma+1)_{m-k}}\\
\hphantom{\delta_{z_{2}}M_{m,n}^{(\beta,\gamma)}(z_{1},z_{2}) =}{}
 - \sum_{k=1}^{n}\frac{(\beta+\gamma+m+1)_{n-k}(\gamma+1)_{m}z_{1}^{m-k}(-z_{2})^{n-k}}{(k-1)!(n-k)!(\gamma+1)_{m-k}}\\
\hphantom{\delta_{z_{2}}M_{m,n}^{(\beta,\gamma)}(z_{1},z_{2}) }{}
 =
 nM_{m,n}^{(\beta,\gamma)}(z_{1},z_{2})-(\gamma+m)M_{m-1,n-1}^{(\beta+1,\gamma)}(z_{1},z_{2}).
\end{gather*}
 Similarly,
\begin{gather*}
\delta_{z_{1}}M_{m,n}^{(\beta,\gamma)}(z_{1},z_{2}) = \sum_{k=0}^{n}\frac{ (\beta+\gamma+m+1 )_{n-k} (\gamma+1 )_{m}z_{1}^{m-k} (-z_{2} )^{n-k} (m-k )}{k! (n-k )! (\gamma+1 )_{m-k}}\\
\hphantom{\delta_{z_{1}}M_{m,n}^{(\beta,\gamma)}(z_{1},z_{2}) }{}
 = m\sum_{k=0}^{n}\frac{ (\beta+\gamma+m+1 )_{n-k} (\gamma+1 )_{m}z_{1}^{m-k} (-z_{2} )^{n-k}}{k! (n-k )! (\gamma+1 )_{m-k}}\\
\hphantom{\delta_{z_{1}}M_{m,n}^{(\beta,\gamma)}(z_{1},z_{2}) =}{}
 - \sum_{k=0}^{n-1}\frac{ (\beta+\gamma+m+1 )_{n-1-k} (\gamma+1 )_{m}z_{1}^{m-1-k} (-z_{2} )^{n-1-k}}{k! (n-1-k )! (\gamma+1 )_{m-1-k}}\\
\hphantom{\delta_{z_{1}}M_{m,n}^{(\beta,\gamma)}(z_{1},z_{2}) }{}
 = mM_{m,n}^{(\beta,\gamma)}(z_{1},z_{2})- (\gamma+m )M_{m-1,n-1}^{(\beta+1,\gamma)}(z_{1},z_{2}),
\\
\delta_{z_{1}}M_{m,n}^{(\beta,\gamma)}(z_{1},z_{2}) = \sum_{k=0}^{n}\frac{ (\beta+\gamma+m+1 )_{n-k+1} (\gamma+1 )_{m}z_{1}^{m-k} (-z_{2} )^{n-k}}{k! (n-k )! (\gamma+1 )_{m-k}} - (\beta+\gamma+n+1 )\\
\hphantom{\delta_{z_{1}}M_{m,n}^{(\beta,\gamma)}(z_{1},z_{2})}{}
 = \sum_{k=0}^{n}\frac{ (\beta+\gamma+m+1 )_{n-k} (\gamma+1 )_{m}z_{1}^{m-k} (-z_{2} )^{n-k}}{k! (n-k )! (\gamma+1 )_{m-k}}\\
 \hphantom{\delta_{z_{1}}M_{m,n}^{(\beta,\gamma)}(z_{1},z_{2})}{}
 = (\beta+\gamma+m+1 )M_{m,n}^{(\beta+1,\gamma)}(z_{1},z_{2})
 - (\beta+\gamma+n+1 )M_{m,n}^{(\beta,\gamma)}(z_{1},z_{2}).\tag*{\qed}
\end{gather*}
 \renewcommand{\qed}{}
\end{proof}

Ismail and Zeng \cite{Ism:Zen1} used the generating function
\begin{gather}
 \sum_{n=0}^{\infty}P_{n}^{(\alpha,\beta)}(1-2x)z^{n}
 =\frac{2^{\alpha+ \beta}}{R (1+z+R )^{\beta} (1-z+R )^{\alpha}},
\label{eq:jacobi-generating-function}
\end{gather}
where
\begin{gather*}
R=\sqrt{\left(1-z\right)^{2}+4xz}, 
\end{gather*}
\cite[Section~140]{Rai} to establish the following generating functions:
\begin{gather*}
\sum_{m,n=0}^\infty M_{m,n}^{(\beta, \gamma)}(z_1, z_2) u^mv^n = \frac{ 2^{\beta +\gamma}}{ \rho} (1 + uv+ \rho)^{-\beta} \,
(1 - uv + \rho)^{-\gamma}\\
\hphantom{\sum_{m,n=0}^\infty M_{m,n}^{(\beta, \gamma)}(z_1, z_2) u^mv^n = }{}
\times \left[\frac{1}{1- 2z_1u/(1+\rho-uv)} + \frac{1}{1- 2v z_2/(1+\rho-uv)} -1\right],
\end{gather*}
and
\begin{gather*}
\sum_{m\ge n \ge 0} M_{m,n}^{(\beta, \gamma)}(z_1, z_2) \frac{u^mv^n}{(m-n)!}
 = \frac{ 2^{\beta +\gamma}}{ \rho}
(1 + uv+ \rho)^{-\beta} \,
(1 - uv + \rho)^{-\gamma}
\exp\left(\frac{2z_1u}{1-uv+\rho}\right),
\end{gather*}
where
\begin{gather}
\label{defrho}
\rho = \big(1-2uv(1-2z_1z_2)+ u^2v^2\big)^{1/2}
\end{gather}
They also verif\/ied the limiting relation
\begin{gather*}
\lim_{\beta \to \infty} M_{m,n}^{(\beta, \gamma)}(z_1, z_2/\beta) = z_1^{m-n} L_n^{(\gamma+m-n)}(z_1z_2) =(-1)^n
Z_{m,n}^{(\gamma)}(z_1, z_2).
\end{gather*}

One can use the generating function \eqref{eq:jacobi-generating-function} to establish the generating relation
\begin{gather*}
\sum_{m\ge n\ge0}u^{m}v^{n}M_{m,n}^{(\beta,\gamma)}(z_{1},z_{2}) = \frac{2^{\beta+\gamma} (1+uv+\rho )^{-\beta} (1-uv+\rho )^{1-\gamma}}{\rho (1-uv-2uz_{1}+\rho )},
\end{gather*}
with $\rho$ as in \eqref{defrho}.

\section[The polynomials $\big\{Z^{(\beta)}_{m,n}(z_1,z_2|q)\big\}$]{The polynomials $\boldsymbol{\big\{Z^{(\beta)}_{m,n}(z_1,z_2|q)\big\}}$}\label{section5}

For $\alpha>-1$, the moment problem associated with $q$-Laguerre
polynomials
\begin{gather*}
L_{n}^{(\alpha)}(x;q) = (q^{\alpha+1};q)_{n}\sum_{k=0}^{n}\frac{q^{\left(\alpha+n-k\right)(n-k)}(-x)^{n-k}}{(q;q)_{k}(q,q^{\alpha+1};q)_{n-k}}.
\end{gather*}
is indeterminate. It is well-known that
\begin{gather*}
\int_{0}^{\infty}L_{m}^{(\alpha)}(x;q)L_{n}^{(\alpha)}(x;q)x^{\alpha}d\mu(x|q)=\zeta_{n}(\alpha)\delta_{m,n},
\end{gather*}
where $\mu(x|q),\zeta_{n}(\alpha)$ are given by
\begin{gather}
d\mu(x|q) =\frac{x^{\alpha}dx}{(-x;q)_{\infty}},\nonumber\\
\zeta_{n}(\alpha) =\Gamma(-\alpha)\Gamma(\alpha+1)\frac{(q^{-\alpha};q)_{\infty}}{(q;q)_{\infty}}\frac{\big(q^{\alpha+1};q\big)_{n}}{q^{n}(q;q)_{n}},
\label{eq:qlaguerre-3}\\
d\mu(x|q) =\sum_{k=-\infty}^{\infty}\frac{x\delta(x-cq^{k})}{(-x;q)_{\infty}},\nonumber\\
\zeta_{n}(\alpha) =\frac{\big(q,-cq^{\alpha+1},-\frac{q^{-\alpha}}{c};q\big)_{\infty}}
{\big(q^{\alpha+1},-c,-\frac{q}{c};q\big)_{\infty}c^{-\alpha-1}}\frac{\big(q^{\alpha+1};q\big)_{n}}{(q;q)_{n}q^{n}}
\label{eq:qlaguerre-4}
\end{gather}
and
\begin{gather}
d\mu(x|q) =\frac{x^{-c}\big({-}\lambda x,-\frac{q}{\lambda x};q\big)_{\infty}dx}{\big(-x,-\lambda q^{c}x,-\frac{q}{\lambda q^{c}x};q\big)_{\infty}},\nonumber\\
\zeta_{n}(\alpha) =\delta_{m,n}\int_{0}^{\infty}\frac{x^{\alpha-c}\big({-}\lambda x,-\frac{q}{\lambda x};q\big)_{\infty}dx}{\big({-}x,-\lambda q^{c}x,-\frac{q}{\lambda q^{c}x};q\big)_{\infty}},\qquad c,\lambda>0,
\label{eq:qlaguerre-5}
\end{gather}
where $c,\lambda>0$. Following the procedure outline in Section~\ref{section2} we let
\begin{gather*}
z_{m,n}^{(\beta)}(z_{1},z_{2}|q)=\begin{cases}
z_{1}^{m-n}L_{n}^{(\beta+m-n)}(z_{1}z_{2};q), & m\ge n,\\
z_{n,m}^{(\beta)}(z_{2},z_{1}|q), & m<n,
\end{cases}
\end{gather*}
then we have
\begin{gather*}
z_{m,n}^{(\beta)}(z_{1},z_{2}|q)=\sum_{k=0}^{n}\frac{q^{\left(\beta+m-k\right)(n-k)}
(q^{\beta+1};q)_{m}z_{1}^{m-k}z_{2}^{n-k}}{(-1)^{n-k}(q;q)_{k}(q;q)_{n-k}(q^{\beta+1};q)_{m-k}},
\\
z_{1}z_{m,n}^{(\beta+1)}(z_{1},z_{2}|\beta)=z_{m+1,n}^{(\beta)}(z_{1},z_{2}|\beta)
\end{gather*}
and
\begin{gather*}
z_{m,n}^{(0)}(z_{1},z_{2}|q)=\frac{(-1)^{n}}{(q;q)_{n}}h_{m,n}(z_{1},z_{2}|q).
\end{gather*}
for $m\ge n$. Applying Theorem~\ref{orthfmn} we obtain the following result
\begin{Theorem}
\label{theorem:2d-q-laguerre-orthogonality} For $m,n,s,t=0,1,\dots$
and $\beta>-1$ we have the following orthogonality relations
\begin{gather*}
\int_{\mathbb{R}^{2}}z_{m,n}^{(\beta)}(z,\bar{z}|q)\overline{z_{s,t}^{(\beta)}(z,\bar{z}|q)}
r^{2\beta}d\mu\big(r^{2}|q\big)d\theta=\zeta_{m\wedge n}(|m-n|+\beta)\delta_{m,s}\delta_{n,t},
\end{gather*}
where $d\mu(x|q)$, $\zeta_{n}(\alpha)$ may be any pairs~\eqref{eq:qlaguerre-3},
\eqref{eq:qlaguerre-4} or~\eqref{eq:qlaguerre-5}.
\end{Theorem}

Applying formulas \eqref{eq3trrG2D} and \eqref{eqrecrel2} in the case
\begin{gather}
c_{k}(n,\alpha)=\frac{\big(q^{\alpha+1};q\big)_{n}q^{(\alpha+n-k)(n-k)}(-1)^{n-k}}
{(q;q)_{k}\big(q,q^{\alpha+1};q\big)_{n-k}}
\end{gather}
we obtain the recurrences stated as a theorem below.
\begin{Theorem}
\label{theorem:q-laguerre-recurrence} For $\beta>-1$ and $m\ge n$ we
have
\begin{gather*}
q^{(m+1+\beta)}z_{2}z_{m+1,n}^{(\beta)}(z_{1},z_{2}) = -(1-q^{n+1})z_{m+1,n+1}^{(\beta)}(z_{1},z_{2})
 + (1-q^{m+\beta+1})z_{m,n}^{(\beta)}(z_{1},z_{2}),\\
q^{n}z_{1}z_{m,n+1}^{(\beta)}(z_{1},z_{2}) = z_{m+1,n}^{(\beta)}(z_{1},z_{2})
-z_{m,n-1}^{(\beta)}(z_{1},z_{2})
\end{gather*}
and
\begin{gather*}
 \big(1+q\big(1-q^{n}-q^{m+\beta}\big)-z_{1}z_{2}q^{\beta+m+n+1}\big)
z_{m,n}^{(\beta)}(z_{1},z_{2})\\
 \qquad{} =\big(1-q^{1+n}\big)z_{m+1,n+1}^{(\beta)}(z_{1},z_{2})+q\big(1-q^{\beta+m}\big)
 z_{m-1,n-1}^{(\beta)}(z_{1},z_{2}).
\end{gather*}
\end{Theorem}

The $q$-Laguerre polynomials satisfy the following $q$-difference
operators
\begin{gather}
D_{q,x}L_{n}^{(\alpha)}(x;q) = -\frac{q^{\alpha+1}}{1-q}L_{n-1}^{(\alpha+1)}(qz;q),
\label{eq:qlaguerre-15}\\
D_{q,x}\big\{ w_{\alpha}(x;q)L_{n}^{(\alpha)}(x;q)\big\} = \frac{1-q^{n+1}}{1-q}
w_{\alpha-1}(x;q)L_{n+1}^{(\alpha-1)}(z;q),\label{eq:qlaguerre-16}\\
w_{\alpha}(x;q)L_{n}^{(\alpha)}(x;q) = \frac{(1-q)^{n}}{(q;q)_{n}}D_{q,x}^{n}
 \{ w_{\alpha+n}(x;q) \} ,\nonumber 
\end{gather}
and
\begin{gather*}
w_{\beta}(x;q)=\frac{x^{\beta}}{(-x;q)_{\infty}}.
\end{gather*}

\begin{Theorem}
For $\beta>-1$ we have
\begin{gather}
\frac{D_{q,z_{1}}\big(z_{1}^{n-m}z_{m,n}^{(\beta)}(z_{1},z_{2})\big)}{z_{1}^{n-m}} =
 \frac{q^{\beta}}{q-1}\frac{z_{2}}{z_{1}}z_{m,n-1}^{(\beta)}(qz_{1},z_{2}),\label{eq:qlaguerre-19}\\
D_{q,z_{2}}\big(z_{m,n}^{(\beta)}(z_{1},z_{2})\big) = \frac{q^{\beta}}{q-1}
z_{m,n-1}^{(\beta)}(qz_{1},z_{2}),\label{eq:qlaguerre-20}
\\
z_{m,n}^{(\beta)}(z_{1},z_{2})=\frac{(1-q)^{m+n}D_{q,z_{1}}^{n}\big\{ z_{2}^{\beta}
D_{q,z_{2}}^{m}\big\{ z_{2}^{-\beta}w_{\beta}(z_{1}z_{2})\big\} \big\} }{(-1)^{m}(q;q)_{n}
w_{\beta}(z_{1}z_{2})},\label{eq:qlaguerre-21}
\\
\frac{z_{2}D_{q,z_{2}}\big\{ w_{\beta}(z_{1}z_{2})z_{m,n}^{(\beta)}(z_{1},z_{2})\big\} }{w_{\beta}(z_{1}z_{2})}
 =\frac{1-q^{\beta}}{1-q}z_{m,n}^{(\beta)}(z_{1},z_{2})-\frac{q^{\beta}z_{2}}{1-q}z_{m+1,n}^{(\beta)}(z_{1},z_{2})
\label{eq:qlaguerre-22}
\end{gather}
for $m\ge n$ and
\begin{gather}
\frac{D_{q,z_{1}}\big\{ w_{\beta}(z_{1}z_{2};q)z_{m,n}^{(\beta)}(z_{1},z_{2})\big\} }{w_{\beta}(z_{1}z_{2};q)} = \frac{1-q^{n+1}}{1-q}z_{m,n+1}^{(\beta)}(z_{1},z_{2}),\label{eq:qlaguerre-23}\\
\frac{D_{q,z_{2}}\big\{ z_{2}^{m-n}w_{\beta}(z_{1}z_{2};q)z_{m,n}^{(\beta)}(z_{1},z_{2})\big\} }{z_{2}^{m-n}w_{\beta}(z_{1}z_{2};q)} = \frac{1-q^{n+1}}{1-q}\frac{z_{1}}{z_{2}}z_{m,n+1}^{(\beta)}(z_{1},z_{2})\label{eq:qlaguerre-24}
\end{gather}
for $m>n$.
\end{Theorem}

\begin{proof}
First we observe that
\begin{gather*}
D_{q,x}\big\{ x^{-\beta}w_{\beta}(x)\big\} =\frac{1}{q-1}x^{-\beta}w_{\beta}(x)
\end{gather*}
 to get
\begin{gather*}
z_{2}^{\beta+m}D_{q,z_{2}}^{m}\big\{ z_{2}^{-\beta}w_{\beta}(z_{1}z_{2})\big\} =\frac{w_{\beta+m}(z_{1}z_{2})}{(q-1)^{m}}
\end{gather*}
and
\begin{gather*}
D_{q,z_{1}}^{n}\big\{ z_{2}^{\beta+m}D_{q,z_{2}}^{m}\big\{ z_{2}^{-\beta}w_{\beta}(z_{1}z_{2})\big\} \big\}
= \frac{D_{q,z_{1}}^{n} \{ w_{\beta+m}(z_{1}z_{2}) \} }{(q-1)^{m}}
 = \frac{z_{2}^{n}}{(q-1)^{m}}D_{q,z_{1}z_{2}}^{n} \{ w_{\beta+m}(z_{1}z_{2}) \} ,
\end{gather*}
or
\begin{gather*}
D_{q,z_{1}}^{n}\big\{ z_{2}^{\beta+m-n}D_{q,z_{2}}^{m}\big\{ z_{2}^{-\beta}w_{\beta}(z_{1}z_{2})\big\} \big\} \\
 \qquad {} = \frac{D_{q,z_{1}z_{2}}^{n} \{ w_{\beta+m}(z_{1}z_{2}) \} }{(q-1)^{m}}
 = \frac{(-1)^{m}(q;q)_{n}}{(1-q)^{m+n}}z_{2}^{m-n}z_{m,n}^{(\beta)}(z_{1},z_{2}),
\end{gather*}
which is \eqref{eq:qlaguerre-21}, \eqref{eq:qlaguerre-19} and \eqref{eq:qlaguerre-20}
are obtained from \eqref{eq:qlaguerre-15}, and \eqref{eq:qlaguerre-23}
and \eqref{eq:qlaguerre-24} are derived from~\eqref{eq:qlaguerre-16}.

From \eqref{eq:qlaguerre-21} we get
\begin{gather*}
 D_{q,z_{2}}\big\{ w_{\beta}(z_{1}z_{2})z_{m,n}^{(\beta)}(z_{1},z_{2})\big\} = \frac{(1-q)^{m+n}}{(-1)^{m}(q;q)_{n}}D_{q,z_{1}}^{n}\big\{ D_{q,z_{2}}\big\{ z_{2}^{\beta}D_{q,z_{2}}^{m}\big\{ z_{2}^{-\beta}w_{\beta}(z_{1}z_{2})\big\} \big\} \big\} \\
 \qquad{}
 = \frac{(1-q)^{m+n-1}\big(1-q^{\beta}\big)}{(-1)^{m}(q;q)_{n}}D_{q,z_{1}}^{n}\big\{ z_{2}^{\beta-1}D_{q,z_{2}}^{m}\big\{ z_{2}^{-\beta}w_{\beta}(z_{1}z_{2})\big\} \big\} \\
 \qquad\quad{}
 + \frac{(1-q)^{m+n}q^{\beta}}{(-1)^{m}(q;q)_{n}}D_{q,z_{1}}^{n}\big\{ z_{2}^{\beta}D_{q,z_{2}}^{m+1}\big\{ z_{2}^{-\beta}w_{\beta}(z_{1}z_{2})\big\} \big\} ,
\end{gather*}
 which gives \eqref{eq:qlaguerre-22}.
\end{proof}

The $q$-Laguerre polynomials satisfy the following second-order dif\/ference
equation
\begin{gather*}
-q^{\alpha}\big(1-q^{n}\big)xy(x) =q^{\alpha}(1+x)y(qx)
 -\big(1+q^{\alpha}(1+x)\big)y(x)+y\big(q^{-1}x\big),
\end{gather*}
 or
\begin{gather*}
 (1+qx)\theta_{q,x}^{2}y(x)+\frac{1-q^{\alpha}-2q^{\alpha+1}x+q^{\alpha+n+1}x}{q^{\alpha}(1-q)}\theta_{q,x}y(x)
 +\frac{qx(1-q^{n})}{(1-q)^{2}}y(x)=0.
\end{gather*}
where $y(x)=L_{n}^{(\alpha)}(x;q)$, $\theta_{q,z}=zD_{q,z}$.
\begin{Theorem}
For $m\ge n$ the function $f=z_{m,n}^{(\beta)}(z_{1},z_{2})$ satisfies the second-order $q$-difference equations
\begin{gather*}
 (1+qz_{1}z_{2})\theta_{q,z_{2}}^{2}f+\frac{q^{n-m-\beta}-1-(2-q^{n})qz_{1}z_{2}}{1-q}\theta_{q,z_{2}}f
 +\frac{qz_{1}z_{2}(1-q^{n})}{(1-q)^{2}}f=0
\end{gather*}
and
\begin{gather*}
(1+qz_{1}z_{2})\theta_{q,z_{1}}^{2}f-\frac{1-q^{\beta+m-n}+(2-q^{\beta+m+1})qz_{1}z_{2}}{1-q}\theta_{q,z_{1}}f
+\frac{q(1-q^{\beta+m})z_{1}z_{2}}{(1-q)^{2}}f=0.
\end{gather*}
\end{Theorem}

Starting with the little $q$-Laguerre \cite{Koe:Swa} or Wall polynomials \cite{Chi}
\begin{gather*}
p_{n}(x;q^{\alpha}|q)=\sum_{k=0}^{n}\frac{(q;q)_{n}q^{\frac{(k-n)(n+k-1)}{2}}(-x)^{n-k}}{(q;q)_{k}(q,q^{\alpha+1};q)_{n-k}}
\end{gather*}
we def\/ine the 2D polynomials $w_{m,n}^{(\beta)}(z_{1},z_{2}|q)$ through
\begin{gather*}
w_{m,n}^{(\beta)}(z_{1},z_{2}|q)=\begin{cases}
z_{1}^{m-n}p_{n}(z_{1}z_{2};q^{\beta+m-n}|q), & m\ge n,\\
w_{n,m}^{(\beta)}(z_{2},z_{1}|q), & m<n,
\end{cases}
\end{gather*}
for $\beta>-1$ and $m,n=0,1,\dots$.
Then for $m\ge n$ f\/ind that
\begin{gather*}
z_{1}w_{m,n}^{(\beta+1)}(z_{1},z_{2}|\beta)=w_{m+1,n}^{(\beta)}(z_{1},z_{2}|\beta),
\\
w_{m,n}^{(\beta)}(z_{1},z_{2}|q)=\sum_{k=0}^{n}\frac{(q;q)_{n}(q^{\beta+1};q)_{m-n}
z_{1}^{m-k}z_{2}^{n-k}q^{-\binom{k}{2}-\binom{n}{2}}}{(-1)^{n-k} (q;q)_{k} (q;q )_{n-k}(q^{\beta+1};q)_{m-k}}.
\end{gather*}
The orthogonality relation for the little $q$-Laguerre polynomials is
\begin{gather*}
 \sum_{k=0}^{\infty}q^{(\alpha+1)k}\big(q^{k+1};q\big)_{\infty}p_{m}\big(q^{k};q^{\alpha}|q\big)p_{n}\big(q^{k};q^{\alpha}|q\big)\\
\qquad {} =\frac{(q;q)_{\infty}q^{(\alpha+1)n}(q;q)_{n}\delta_{m,n}}{(q^{\alpha+1};q)_{\infty}(q^{\alpha+1};q)_{n}},\qquad m,n=0,1,\dots.
\end{gather*}
This leads to the orthogonality relation $\{w_{m,n}^{(\beta)}(z_{1},z_{2}|q)\}_{m,n=0}^{\infty}$ in the form

\begin{Theorem}
\label{theorem:wall-1} For $m,n,s,t=0,1,\dots$ and $\beta>-1$ we have
the following orthogonality relations
\begin{gather*}
\int_{\mathbb{R}^{2}}w_{m,n}^{(\beta)}(z,\bar{z}|q)\overline{w_{s,t}^{(\beta)}(z,\bar{z}|q)}r^{2\beta}
d\mu\big(r^{2}|q\big)d\theta=\zeta_{m\wedge n}(|m-n|+\beta)\delta_{m,s}\delta_{n,t},
\end{gather*}
where $z=re^{i\theta}$,
\begin{gather*}
d\mu(r|q)=\sum_{k=0}^{\infty}x(xq;q)_{\infty}\delta\big(x-q^{k}\big),
\end{gather*}
and
\begin{gather*}
\zeta_{n}(\alpha)=\frac{(q;q)_{\infty}q^{(\alpha+1)n}(q;q)_{n}}
{(q^{\alpha+1};q)_{\infty}(q^{\alpha+1};q)_{n}}.
\end{gather*}
\end{Theorem}

The function $y(x)=p_{n}(x;q^{\alpha}|q)$ satisf\/ies the $q$-dif\/ference equation
\begin{gather*}
q^{\alpha}y\big(q^{2}z\big)-\big(1+q^{\alpha}-zq^{1-n}\big)y(qz)+(1-qz)y(z)=0
\end{gather*}
or equivalently
\begin{gather*}
q^{\alpha+n-1}\theta_{q,z}^{2}y(z)+\frac{q^{n-1}+q^{\alpha+n-1}-z}{1-q}\theta_{q,z}y(z)
+\frac{z(1-q^{n})}{(1-q)^{2}}y(z)=0,
\end{gather*}
where $\theta_{q,z}=zD_{q,z}$. This leads to the following theorem.
 \begin{Theorem}
\label{theorem:wall-2}For $\beta>-1$, $m\ge n$ the polynomial
$w_{m,n}^{(\beta)}(z_{1},z_{2}|q)$,
satisfies the $q$-partial diffe\-ren\-ce equations
\begin{gather*}
q^{\beta+m-1}\theta_{q,z_{2}}^{2}f+\frac{q^{n-1}-q^{\beta+m-1}-z_{1}z_{2}}{1-q}\theta_{q,z_{2}}f+
\frac{(1-q^{n})z_{1}z_{2}}{(1-q)^{2}}f=0, 
\end{gather*}
and
\begin{gather*}
q^{n}\theta_{q,z_{1}}^{2}f-\frac{q^{n}-q^{\beta+m}+qz_{1}z_{2}}{1-q}\theta_{q,z_{1}}f
+ \frac{q(1-q^{\beta+m})z_{1}z_{2}}{(1-q)^{2}}f=0.
\end{gather*}
\end{Theorem}
From
\begin{gather}
D_{q,x}p_{n}(x;q^{\alpha}|q)=-\frac{q^{1-n}(1-q^{n})}{(1-q)(1-q^{\alpha+1})}p_{n-1}\big(x;q^{\alpha+1}|q\big),\nonumber
\\
\frac{D_{q^{-1},x}\left\{ w(x;q^{\alpha})p_{n}(x;q^{\alpha}|q)\right\} }{w(x;\alpha-1)}=\frac{(1-q^{\alpha})p_{n+1}(x;q^{\alpha-1}|q)}{q^{\alpha-1}(1-q)}\label{eq:wall-14}
\end{gather}
 and
\begin{gather*}
w(x;\alpha|q)p_{n}(x;q^{\alpha}|q)=\frac{q^{n\alpha+\binom{n}{2}}(1-q)^{n}}{(q^{\alpha+1};q)_{n}}D_{q^{-1},x}^{n}w(x;\alpha+n|q),
\end{gather*}
where
\begin{gather*}
w(x;\alpha)=(qx;q)_{\infty}x^{\alpha}
\end{gather*}
 we get the following relations:
\begin{Theorem}
For $\beta>-1$ and $m\ge n$ we have
\begin{gather*}
\frac{D_{q,z_{1}}\big\{ z_{1}^{n-m}w_{m,n}^{(\beta)}(z_{1},z_{2}|q)\big\} }{z_{1}^{n-m}}=-\frac{z_{2}}{z_{1}}\frac{q^{1-n}(1-q^{n})w_{m,n-1}^{(\beta)}(z_{1},z_{2}|q)}{(1-q)(1-q^{\beta+m-n+1})},
\\
D_{q,z_{2}}\big\{ w_{m,n}^{(\beta)}(z_{1},z_{2}|q)\big\} =-\frac{q^{1-n}(1-q^{n})w_{m,n-1}^{(\beta)}(z_{1},z_{2}|q)}{(1-q)(1-q^{\beta+m-n+1})}
\end{gather*}
and
\begin{gather}
 w(z_{1}z_{2};\beta|q)w_{m,n}^{(\beta)}(z_{1},z_{2}|q)
 =\frac{q^{m(n-1)+n(\beta-1)-\binom{n}{2}}(1-q)^{m+n}}{(-1)^{m}(q^{\beta+m-n+1};q)_{n}}\nonumber\\
\hphantom{w(z_{1}z_{2};\beta|q)w_{m,n}^{(\beta)}(z_{1},z_{2}|q)=}{}
 \times D_{q^{-1},z_{1}}^{n}\big\{ (z_{1}z_{2})^{\beta}D_{q^{-1},z_{2}}^{m}\big\{ (z_{1}z_{2})^{-\beta}w(z_{1}z_{2};\beta|q)\big\} \big\} .
\label{eq:wall-19}
\end{gather}
For $m>n$ we have
\begin{gather}
\frac{D_{q^{-1},z_{1}}\big\{ w(z_{1}z_{2};\beta|q)w_{m,n}^{(\beta)}(z_{1},z_{2}|q)\big\} }
{w(z_{1}z_{2};\beta|q)}=\frac{(1-q^{\beta+m-n})w_{m,n+1}^{(\beta)}(z_{1},z_{2}|q)}
{q^{\beta+m-n-1}(1-q)},\label{eq:wall-20}
\\
\frac{D_{q^{-1},z_{2}}\big\{ z_{2}^{m-n}w(z_{1}z_{2};\beta|q)w_{m,n}^{(\beta)}(z_{1},z_{2}|q)\big\} }
{z_{2}^{m-n}w(z_{1}z_{2};\beta|q)}=\frac{z_{1}}{z_{2}}\frac{(1-q^{\beta+m-n})w_{m,n+1}^{(\beta)}
(z_{1},z_{2}|q)}{q^{\beta+m-n-1}(1-q)},
\label{eq:wall-21}
\\
 \frac{z_{2}D_{q^{-1},z_{2}}\big\{ w(z_{1}z_{2};\beta|q)w_{m,n}^{(\beta)}(z_{1},z_{2}|q)\big\} }{w_{m,n}^{(\beta)}(z_{1},z_{2}|q)}\nonumber\\
\qquad{}
=\frac{1-q^{\beta}}{1-q}w_{m,n}^{(\beta)}(z_{1},z_{2}|q)-\frac{q^{\beta+1-n}}{1-q}\frac{1-q^{\beta+m+1}}{1-q^{\beta+m-n+1}}z_{2}w_{m+1,n}^{(\beta)}(z_{1},z_{2}|q)
\label{eq:wall-22}
\end{gather}
\end{Theorem}

\begin{proof}
Observe that
\begin{gather*}
D_{q^{-1},z_{1}}^{n}=z_{2}^{n}D_{q^{-1},z_{1}z_{2}}^{n}
\end{gather*}
and
\begin{gather*}
z_{1}^{m}w(z_{1}z_{2};\beta|q)=\big(1-q^{-1}\big)^{m}(z_{1}z_{2})^{\beta}D_{q^{-1},z_{2}}^{m}
\big\{ (z_{1}z_{2})^{-\beta}w(z_{1}z_{2};\beta|q)\big\} ,
\end{gather*}
 then
\begin{gather*}
w(z_{1}z_{2};\beta|q)w_{m,n}^{(\beta)}(z_{1},z_{2}|q) = \frac{q^{n(\beta+m)
- \binom{n+1}{2}}(1-q)^{n}}{(q^{\beta+m-n+1};q)_{n}} D_{q^{-1},z_{1}}^{n}\big(z_{1}^{m}w(z_{1}z_{2};\beta|q)\big)\\
\hphantom{w(z_{1}z_{2};\beta|q)w_{m,n}^{(\beta)}(z_{1},z_{2}|q)}{}
 = \frac{q^{m(n-1)+n(\beta-1)-\binom{n}{2}}(1-q)^{m+n}}{(-1)^{m}(q^{\beta+m-n+1};q)_{n}}\\
 \hphantom{w(z_{1}z_{2};\beta|q)w_{m,n}^{(\beta)}(z_{1},z_{2}|q)=}{}
 \times D_{q^{-1},z_{1}}^{n}\big\{ (z_{1}z_{2})^{\beta}D_{q^{-1},z_{2}}^{m}
 \big\{ (z_{1}z_{2})^{-\beta}w(z_{1}z_{2};\beta|q)\big\} \big\} ,
\\
w(z_{1}z_{2};\beta|q)w_{m,n}^{(\beta)}(z_{1},z_{2}|q)
=\frac{q^{m(n-1)+n(\beta-1)-\binom{n}{2}}(1-q)^{m+n}}{(-1)^{m}(q^{\beta+m-n+1};q)_{n}}\\
\hphantom{w(z_{1}z_{2};\beta|q)w_{m,n}^{(\beta)}(z_{1},z_{2}|q)=}{}
 \times D_{q^{-1},z_{1}}^{n}\big\{ (z_{1}z_{2})^{\beta}D_{q^{-1},z_{2}}^{m}
 \big\{ (z_{1}z_{2})^{-\beta}w(z_{1}z_{2};\beta|q)\big\} \big\} ,
\end{gather*}
which gives \eqref{eq:wall-19}, \eqref{eq:wall-20} and \eqref{eq:wall-21}
are obtained from \eqref{eq:wall-14} by direct computation.

From \eqref{eq:wall-19} we have
\begin{gather*}
 D_{q^{-1},z_{2}}\big\{ w(z_{1}z_{2};\beta|q)w_{m,n}^{(\beta)}(z_{1},z_{2}|q)\big\} =\frac{q^{m(n-1)+n(\beta-1)-\binom{n}{2}}(1-q)^{m+n}}{(-1)^{m}(q^{\beta+m-n+1};q)_{n}} \\
 \qquad\quad{}
\times
 D_{q^{-1},z_{1}}^{n}\big\{ D_{q^{-1},z_{2}}\big\{ (z_{1}z_{2})^{\beta}D_{q^{-1},z_{2}}^{m}
 \big\{ (z_{1}z_{2})^{-\beta}w(z_{1}z_{2};\beta|q)\big\} \big\} \big\} \\
\qquad{}
 =\frac{q^{m(n-1)+n(\beta-1)-\binom{n}{2}}(1-q)^{m+n}}{(-1)^{m}(q^{\beta+m-n+1};q)_{n}}
 q^{\beta}
 D_{q^{-1},z_{1}}^{n}\big\{ (z_{1}z_{2})^{\beta}D_{q^{-1},z_{2}}^{m+1}
 \big\{ (z_{1}z_{2})^{-\beta}w(z_{1}z_{2};\beta|q)\big\} \big\} \\
\qquad\quad{}
 +\frac{q^{m(n-1)+n(\beta-1)+1-\beta-\binom{n}{2}}(1-q)^{m+n-1}(1-q^{\beta})}{(-1)^{m}
 (q^{\beta+m-n+1};q)_{n}}\\
\qquad\quad{}
 \times z_{2}^{-1}D_{q^{-1},z_{1}}^{n}\big\{ (z_{1}z_{2})^{\beta}D_{q^{-1},z_{2}}^{m}
 \big\{ (z_{1}z_{2})^{-\beta}w(z_{1}z_{2};\beta|q)\big\} \big\} \\
\qquad{}
 =-\frac{q^{\beta+1-n}}{1-q}\frac{1-q^{\beta+m+1}}{1-q^{\beta+m-n+1}}w(z_{1}z_{2};\beta|q)
 w_{m+1,n}^{(\beta)}(z_{1},z_{2}|q)\\
\qquad\quad{}
 +\frac{(1-q^{\beta})}{z_{2}(1-q)}w(z_{1}z_{2};\beta|q)w_{m,n}^{(\beta)}(z_{1},z_{2}|q),
\end{gather*}
which gives \eqref{eq:wall-22}.
\end{proof}

Applying \eqref{eq3trrphi}--\eqref{eqcn} and \eqref{eq2.21} to
\begin{gather*}
c_{k}(n,\alpha)=\frac{(q;q)_{n}q^{\frac{(k-n)(n+k-1)}{2}}(-1)^{n-k}}{(q;q)_{k}(q,q^{\alpha+1};q)_{n-k}}
\end{gather*}
we get the following recurrences:
\begin{Theorem}
For $\beta>-1$, $m\ge n$ we have
\begin{gather*}
\frac{z_{2}w_{m+1,n}^{(\beta)}(z_{1},z_{2}|q)}{q^{n}(1-q^{m-n+\beta+1})}
 =w_{m,n}^{(\beta)}(z_{1},z_{2}|q)-w_{m+1,n+1}^{(\beta)}(z_{1},z_{2}|q),
\\
 z_{1}\big(1-q^{m-n+\beta+1}\big)w_{m,n}^{(\beta)}(z_{1},z_{2}|q)\\
 \qquad{}
 =\big(1-q^{m+\beta+1}\big)w_{m+1,n}^{(\beta)}(z_{1},z_{2}|q)-q^{m-n+\beta+1}(1-q^{n})
 w_{m,n-1}^{(\beta)}(z_{1},z_{2}|q),
\end{gather*}
 and
\begin{gather*}
 \big(q^{n}+q^{m+\beta}\big(1-q^{n}-q^{n+1}\big)-z_{1}z_{2}\big)
w_{m,n}^{(\beta)}(z_{1},z_{2}|q)\\
\qquad{} =q^{n}\big(1-q^{m+\beta+1}\big)w_{m+1,n+1}^{(\beta)}(z_{1},z_{2}|q)+q^{m+\beta}(1-q^{n})
 w_{m-1,n-1}^{(\beta)}(z_{1},z_{2}|q).
\end{gather*}
\end{Theorem}

 \section[The polynomials $\big\{M_n^{(\alpha,\beta)}(z_1,z_2|q)\big\}$]{The polynomials $\boldsymbol{\big\{M_n^{(\alpha,\beta)}(z_1,z_2|q)\big\}}$}\label{section6}

For $\alpha,\gamma>-1$, the Little $q$-Jacobi polynomials
\begin{gather*}
p_{n}\big(x;q^{\alpha},q^{\gamma}|q\big) =\frac{(q;q)_{n}}{q^{\binom{n}{2}}}
\sum_{k=0}^{n}\frac{(q^{\alpha+\gamma+n+1};q)_{n-k}q^{\binom{k}{2}}(-x)^{n-k}}{(q;q)_{k}
(q,q^{\alpha+1};q)_{n-k}}
\end{gather*}
satisfy the orthogonality relation \cite{Koe:Swa}
\begin{gather*}
 \sum_{k=0}^{\infty}\frac{(q^{k+1};q)_{\infty}q^{k(\alpha+1)}}
{(q^{\gamma+k+1};q)_{\infty}}p_{m}(q^{k};q^{\alpha},q^{\gamma}|q)
p_{n}\big(q^{k};q^{\alpha},q^{\gamma}|q\big)\\
\qquad{} =\frac{(q,q^{\alpha+\gamma+n+1};q)_{\infty}(q;q)_{n}q^{n(\alpha+1)}\delta_{m,n}}
 {(q^{\alpha+1},q^{\gamma+n+1};q)_{\infty}(q^{\alpha+1};q)_{n}(1-q^{\alpha+\gamma+2n+1})}.
\end{gather*}
This leads to the choice
\begin{gather}
d\mu(x;\gamma|q) =\sum_{k=0}^{\infty}\frac{(qx;q)_{\infty}x}
 {(q^{\gamma+1}x;q)_{\infty}}\delta\big(x-q^{k}\big),\nonumber\\
\zeta_{n}(\alpha,\gamma|q) =\frac{(q,q^{\alpha+\gamma+n+1};q)_{\infty}(q;q)_{n}
q^{n(\alpha+1)}}{(q^{\alpha+1},q^{\gamma+n+1};q)_{\infty}(q^{\alpha+1};q)_{n}
(1-q^{\alpha+\gamma+2n+1})}.
\label{eq:qjacobi-3}
\end{gather}
Following our general construction we def\/ine the polynomials
\begin{gather*}
M_{m,n}^{(\beta,\gamma)}(z_{1},z_{2}|q)=\begin{cases}
z_{1}^{m-n}p_{n}\big(z_{1}z_{2},q^{\beta+m-n},q^{\gamma}|q\big), & m\ge n,\\
M_{n,m}^{(\beta,\gamma)}(z_{2},z_{1}|q), & m<n.
\end{cases}
\end{gather*}
Then we have the explicit form
\begin{gather*}
 M_{m,n}^{(\beta,\gamma)}(z_{1},z_{2}|q) =\frac{(q;q)_{n}(q^{\beta+1};q)_{m-n}}
{(q^{\beta+\gamma+1};q)_{m}q^{\binom{n}{2}}}\sum_{k=0}^{n}\frac{(q^{\beta+\gamma+1};q)_{m+n-k}q^{\binom{k}{2}}z_{1}^{m-k}
 (-z_{2})^{n-k}}{(q;q)_{k}(q;q)_{n-k}(q^{\beta+1};q)_{m-k}},
\\
z_{1}M_{m,n}^{(\beta+1,\gamma)}(z_{1},z_{2}|q)=M_{m+1,n}^{(\beta,\gamma)}(z_{1},z_{2}|q)
\end{gather*}
for $m\ge n$. Applying Theorem~\ref{orthfmn} we obtain:
\begin{Theorem}
\label{theorem:qjacobi-1} For $m,n,s,t=0,1,\dots$ and $\beta,\gamma>-1$
we have the following orthogonality relation
\begin{gather*}
 \int_{\mathbb{R}^{2}}M_{m,n}^{(\beta,\gamma)}(z,\bar{z}|q)
\overline{M_{s,t}^{(\beta,\gamma)}(z,\bar{z}|q)}r^{2\beta}d\mu\big(r^{2};\gamma|q\big)d\theta
 =\zeta_{m\wedge n}(\left|m-n\right|+\beta,\gamma|q)\delta_{m,s}\delta_{n,t},
\end{gather*}
where $d\mu(x;\gamma|q)$, $\zeta_{n}(\alpha,\gamma|q)$ are given in~\eqref{eq:qjacobi-3}.
\end{Theorem}

Next we apply formulas \eqref{eq3trrphi}--\eqref{eqcn} and \eqref{eq2.21} to
\begin{gather*}
c_{k}(n,\alpha)=\frac{(q;q)_{n}(q^{\alpha+\gamma+n+1};q)_{n-k}(-1)^{n-k}}{(q;q)_{k}
(q,q^{\alpha+1};q)_{n-k}q^{\binom{n}{2}-\binom{k}{2}}},
\end{gather*}
and obtain the following recurrences:

\begin{Theorem}
\label{theorem:qjacobi-2} For $\beta,\gamma>-1$ and $m\ge n$ we have
\begin{gather*}
 \frac{(1-q^{\beta+\gamma+m+n+2})}{q^{n}(1-q^{\beta+m-n+1})}z_{2}
M_{m+1,n}^{(\beta,\gamma)}(z_{1},z_{2}|q)
 =-M_{m+1,n+1}^{(\beta,\gamma)}(z_{1},z_{2}|q)+M_{m,n}^{(\beta,\gamma)}(z_{1},z_{2}|q),
\\
 z_{1}M_{m,n}^{(\beta,\gamma)}(z_{1},z_{2}|q)
 =\frac{(1-q^{\beta+m+1})(1-q^{\beta+\gamma+m+1})}{(1-q^{\beta+m-n+1})(1-q^{\beta+\gamma
 + m+n+1})}M_{m+1,n}^{(\beta,\gamma)}(z_{1},z_{2}|q)\\
 \hphantom{z_{1}M_{m,n}^{(\beta,\gamma)}(z_{1},z_{2}|q) =}{}
 +\frac{q^{m-n+\beta+1}(1-q^{n})(1-q^{\gamma+n})}{(1-q^{\beta+m-n+1})(1-q^{\beta+\gamma+m
 +n+1})}M_{m,n-1}^{(\beta,\gamma)}(z_{1},z_{2}|q).
\end{gather*}
\end{Theorem}

The polynomials $\{p_{n}(x;q^{\alpha},q^{\gamma}|q)\}$ satisfy the
following second-order dif\/ference equation
\begin{gather*}
 \big(1+q^{\alpha}-\big(q^{1-n}+q^{n+\alpha+\gamma+2}\big)x\big)y(xq) =q^{\alpha}(1-q^{\gamma+2}x)y\big(q^{2}x\big)+(1-qx)y(x)
\end{gather*}
 or
\begin{gather*}
 \big(1-q^{\gamma+2}x\big)\theta_{q,x}^{2}y(x)+\frac{q^{-\alpha}
- 1-q^{1-n-\alpha}x-q^{\gamma+2}\big(2-q^{n}\big)x}{1-q}\theta_{q,x}y(x)\\
\qquad{}
 -\frac{1+q^{1-\alpha}x-(q^{1-n-\alpha}+q^{\gamma+n+2})x}{(1-q)^{2}}y(x)=0,
\end{gather*}
where $y(x)=p_{n}(x;q^{\alpha},q^{\gamma}|q$), $\theta_{q,z}=zD_{q,z}$.
\begin{Theorem}
For $m\ge n$ and $\beta,\gamma>-1$, the function $f=M_{m,n}^{(\beta,\gamma)}(z_{1},z_{2})$
satisfies
\begin{gather*}
 \theta_{q,z_{2}}^{2}f+\left\{ \frac{q^{n-m-\beta}-1-q^{1-m-\beta}z_{1}z_{2}
-q^{\gamma+2}(2-q^{n})z_{1}z_{2}}{(1-q)(1-q^{\gamma+2}z_{1}z_{2})}\right\} \theta_{q,z_{2}}f\\
\qquad{} -\frac{1+(q^{n+1-m-\beta}-q^{1-m-\beta}-q^{\gamma+n+2})z_{1}z_{2}}{(1-q)^{2}
 (1-q^{\gamma+2}z_{1}z_{2})}f=0
\end{gather*}
and
\begin{gather*}
 \theta_{q,z_{1}}^{2}f+\frac{z_{1}z_{2}\left(2q^{\gamma+2}-q^{1-n}-q^{m+\beta
+ \gamma+2}\right)+q^{m-n+\beta}-1}{(1-q)(1-q^{\gamma+2}z_{1}z_{2})}\theta_{q,z_{1}}f\\
\qquad{} +\frac{\left(q^{1-n}-q^{1+m-n+\beta}+q^{\gamma+2}(q^{m+\beta}
 + q^{2(m-n+\beta)}-1)\right)z_{1}z_{2}-q^{2(m-n+\beta)}}{(1-q)^{2}(1-q^{\gamma+2}z_{1}z_{2})}f=0.
\end{gather*}
\end{Theorem}

The polynomials $\{p_{n}(x;q^{\alpha},q^{\gamma}|q)\}$ satisfy the
following difference equations
\begin{gather}
 p_{n}\big(x;q^{\alpha},q^{\gamma}|q\big)-p_{n}\big(qx;q^{\alpha},q^{\gamma}|q\big) \nonumber\\
 \qquad{}
 =-\frac{q^{1-n}(1-q^{n})(1-q^{\alpha+\gamma+n-1})}{1-q^{\alpha+1}}xp_{n-1}\big(x;q^{\alpha
 +1},q^{\gamma+1}|q\big),
\label{eq:qjacobi-15}
\\
 D_{q^{-1},x}\big(w(x;q^{\alpha},q^{\gamma}|q)
p_{n}(x;q^{\alpha},q^{\gamma}|q\big)
 =\frac{(1-q^{\alpha})w(x;q^{\alpha-1},q^{\gamma-1}|q)}
 {q^{\alpha-1}(1-q)}p_{n+1}\big(x;q^{\alpha-1},q^{\gamma-1}|q\big),\nonumber
\end{gather}
 where
\begin{gather}
w\big(x;q^{\alpha},q^{\gamma}|q\big)=\frac{(qx;q)_{\infty}x^{\alpha}}{(q^{\gamma+1}x;q)_{\infty}}.
\label{eq:qjacobi-17}
\end{gather}

\begin{Theorem}
For $\beta,\gamma>-1$ we have
\begin{gather*}
 M_{m,n}^{(\beta,\gamma)}(z_{1},z_{2}|q) -
M_{m,n}^{(\beta,\gamma)}(z_{1},qz_{2}|q)
 =-\frac{q^{1-n}(1-q^{n})(1-q^{m+\beta+\gamma-1})}
 {1-q^{m-n+\beta}}z_{2}M_{m,n-1}^{(\beta,\gamma+1)}(z_{1},z_{2}|q)
\end{gather*}
for $m\ge n$, and
\begin{gather*}
 D_{q^{-1},z_{1}}\big(w(z_{1}z_{2};q^{\beta},q^{\gamma}|q)
M_{m,n}^{(\beta,\gamma)}(z_{1},z_{2}|q)\big)\\
 \qquad{} =\frac{(1-q^{m-n+\beta})w(z_{1}z_{2};q^{\beta},q^{\gamma-1}|q)}
 {q^{m-n+\beta}(1-q)}M_{m,n+1}^{(\beta,\gamma-1)}(z_{1},z_{2}|q),
\\
 z_{2}^{n-m}D_{q^{-1},z_{2}}
\big(z_{2}^{m-n}w(z_{1}z_{2};q^{\beta},q^{\gamma}|q)M_{m,n}^{(\beta,\gamma)}(z_{1},z_{2}|q)\big)
\\
\qquad{} =\frac{z_{1}(1-q^{m-n+\beta})w(z_{1}z_{2};q^{\beta},q^{\gamma-1}|q)}{z_{2}q^{m-n+\beta-1}(1-q)}
 M_{m,n+1}^{(\beta,\gamma-1)}(z_{1},z_{2}|q)
\end{gather*}
for $m\ge n+1$.
\end{Theorem}

The above recurrences are obtained by applying the def\/inition of
$M_{m,n}^{(\beta,\gamma)}(z_{1},z_{2}|q)$
to~\eqref{eq:qjacobi-15} and~\eqref{eq:qjacobi-17}.

 \section{Polynomial solutions to dif\/ferential equations}\label{section7}

 In this section we study polynomial solutions to partial dif\/ferential equations. We are looking for polynomial solutions to the second-order partial dif\/ferential equation~\eqref{eqPDE1}. The results of this section should be contrasted with those in~\cite{Fin:Kam}.

 Let $f = \sum\limits_{m,n=0}^\infty a_{j,k}z_1^j z_2^k$ and substitute in~\eqref{eqPDE1}.

\begin{Theorem} \label{theorem7.1} The partial differential equation~\eqref{eqPDE1}, namely,
 \begin{gather*}
 \partial_{\partial z_1} \partial_{\partial z_2} f +
 \left( \frac{\beta-z_1z_2}{z_1} \right)
 \partial_{\partial z_2} f = -n f
 \end{gather*}
has a power series solution
\begin{gather}
f(z_1, z_2) = \sum_{j,k=0}^\infty a_{j,k} z_1^j z_2^k
\label{eqpowerser}
\end{gather}
 if and only if
 \begin{gather*}
 f(z_1, z_2) = \sum_{j=1}^\infty a_{j,0} \frac{n! z_1^j }{(\beta+j+1)_n} L_n^{(\beta+j)}(z_1z_2)
 + \sum_{k=0}^na_{0,k} z_2^k {}_2F_1\left( \left. \begin{matrix}
 -n+k, 1
 \\
k+1, \beta+1
 \end{matrix} \right| z_1z_2\right).
 \end{gather*}
 \end{Theorem}

 \begin{proof}
Substitute the power series \eqref{eqpowerser} for $f$ in \eqref{eqPDE1} and equate coef\/f\/icients of
like powers of $z_1$ and $z_2$ to f\/ind that
\begin{gather*}
a_{j,k} = \frac{k-1-n}{k(\beta +j)} a_{j-1,k-1}, \qquad j,k >0.
\end{gather*}
 We iterate this and f\/ind that
\begin{gather*}
a_{j,k} = \begin{cases} \dfrac{(-n)_k a_{j-k,0}}{k!(\beta+j-k+1)_k}, & j \ge k,\vspace{1mm}\\
\dfrac{(k-j-n)_j (k-j)!}{k!(\beta+1)j} a_{0,k-j}, & j \le k.
\end{cases}
\end{gather*}
This proves the theorem.
\end{proof}

 \begin{Theorem}
 In order for the equation
 \begin{gather*}
 \partial_{z_1}\partial_{z_2}f + \left(\frac{\beta}{z_1}- z_2\right) \partial_{z_2}f = \lambda f
\end{gather*}
to have a polynomials solution in $z_2$, it is necessary and sufficient that $\lambda =-n$, $n=0,1, 2, \dots$,
in which case the function~$f$ will be given as in Theorem~{\rm \ref{theorem7.1}}.
\end{Theorem}

 \section{A biorthogonal system}\label{section8}

\begin{Theorem}
\label{theorem:orthogonal system 2} Given three sets $A$, $B$, $C$ and
a function $\rho$ from $C$ to $B$. For each $b\in B$, let $\{ S,\mathcal{F}_{S},d\mu(x)\} $
and $\{ T,\mathcal{F}_{T},d\nu(x)\} $ be two probability
spaces such that $\{ \varphi_{a}(x;b)f(x,b)\} _{a\in A}\subset L^{2}(S,d\mu(x))$
and $\{ \psi_{c}(y)\} _{c\in C}\subset L^{2}(T,d\nu(y))$,
\begin{gather*}
\int_{S}\varphi_{a_{1}}(x;b)\overline{\varphi_{a_{2}}(x;b)}
 |f(x,b) |^{2}d\mu(x)=\zeta_{a_{1}}(b)\delta_{a_{1},a_{2}}
\end{gather*}
and
\begin{gather*}
\int_{T}\psi_{c_{1}}(y)\overline{\psi_{c_{2}}(y)}d\nu(y)=\eta_{c_{1}}\delta_{c_{1},c_{2}}
.
\end{gather*}
Then
\begin{gather*}
\int_{S\times T}\Phi_{a_{1},c_{1}}(x,y)\overline{\Phi_{a_{2},c_{2}}(x,y)}d\mu(x)\nu(y)=\zeta_{a_{1}}
(\rho(c_{1}))\eta_{c_{1}}\delta_{a_{1},a_{2}}\delta_{c_{1},c_{2}},
\end{gather*}
where
\begin{gather*}
\Phi_{a_{1},c_{1}}(x,y)=\varphi_{a_{1}}(x;\rho(c_{1}))\psi_{c_{1}}(y)f(x;\rho(c_{1})).
\end{gather*}
 \end{Theorem}
\begin{proof}
Observe that for each $\left(a,c\right)\in A\times C$, $\Phi_{a,c}(x,y)\in L^{2} (S\times T,d\mu(x)d\nu(y) )$
by Fubini's theorem. Then for $a_{1},a_{2}\in A$, $c_{1},c_{2}\in C$
we have $\Phi_{a_{1},c_{1}}(x,y)
\overline{\Phi_{a_{2},c_{2}}(x,y)}\in L (S\times T,d\mu(x)d\nu(y) )$
by applying Cauchy--Schwartz inequality. By applying Fubini's theorem
we get
\begin{gather*}
\int_{S\times T}
\Phi_{a_{1},c_{1}}(x,y)\overline{\Phi_{a_{2},c_{2}}(x,y)}d\mu(x)\nu(y)\\
 =\int_{S}\varphi_{a_{1}}(x;\rho(c_{1}))
 f(x,\rho(c_{1}))\overline{\varphi_{a_{2}}(x;\rho(c_{2}))f(x,\rho(c_{2}))}d\mu(x)
\int_{T}\psi_{c_{1}}(y)\overline{\psi_{c_{2}}(y)}d\nu(y)\\
 =\eta_{c_{1}}\delta_{c_{1},c_{2}}
 \int_{S}\varphi_{a_{1}}(x;\rho(c_{1}))\overline{\varphi_{a_{2}}(x;\rho(c_{1}))}
 \left|f(x,\rho(c_{1}))\right|^{2}d\mu(x)
 =\zeta_{a_{1}}(\rho(c_{1}))\eta_{c_{1}}\delta_{a_{1},a_{2}}\delta_{c_{1},c_{2}}.\tag*{\qed}
\end{gather*}
\renewcommand{\qed}{}
\end{proof}

For $a,b,c,d\in(-1,1)$, the Askey--Wilson polynomials
\begin{gather*}
p_{n}(\cos\theta;a,b,c,d|q)= {}_{4}\phi_{3}\left(\begin{matrix}
q^{-n},abcdq^{n-1},ae^{i\theta},ae^{-i\theta} \\
ab,ac,ad
\end{matrix};q,q\right),
\end{gather*}
 satisfy
\begin{gather*}
 \int_{-1}^{1}p_{m}(x;a,b,c,d|q)p_{n}(x;a,b,c,d|q)w(x;a,b,c,d|q)dx \\ 
\qquad{} =\frac{2\pi(abcdq^{2n};q)_{\infty}(abcdq^{n-1};q)_{n}\delta_{m,n}}{(q^{n+1};q)_{\infty}
 (abq^{n},acq^{n},adq^{n},bcq^{n},bdq^{n},cdq^{n};q)_{\infty}},
\end{gather*}
 where
\begin{gather*}
w(x;a,b,c,d|q)=\frac{h(x,1)h(x,\sqrt{q})h(x,-1)h(x,-\sqrt{q})}{h(x,a)h(x,b)h(x,c)h(x,d)\sqrt{1-x^{2}}},
\\
h(x,a)=\prod_{k=0}^{\infty}\big(1-2axq^{k}+a^{2}q^{2k}\big),\qquad x=\cos\theta.
\end{gather*}

\begin{Theorem}\label{theorem:tensor-askey-wilson} Let
\begin{gather*}
 u_{j,k}^{(\gamma,\delta)}(x,y;a_{1},b_{1},c_{1},d_{1};a_{2},b_{2},c_{2},d_{2}|q)
= p_{j}\big(x;a_{1},b_{1},c_{1},d_{1}q^{\gamma k+\delta}|q\big)p_{k}(y;a_{2},b_{2},c_{2},d_{2}|q),
\\
 v_{j,k}^{(\gamma,\delta)}(x,y;a_{1},b_{1},c_{1},d_{1};a_{2},b_{2},c_{2},d_{2}|q)
= \frac{p_{j}(x;a_{1},b_{1},c_{1},d_{1}q^{\gamma k+\delta}|q)p_{k}(y;a_{2},b_{2},c_{2},d_{2}|q)}{h(x,d_{1}q^{\gamma k+\delta})},
\\
 w(x,y;a_{1},b_{1},c_{1};a_{2},b_{2},c_{2},d_{2}|q)\\
\qquad {}= \frac{h(x,1)h(x,\sqrt{q})h(x,-1)h(x,-\sqrt{q})}{h(x,a_{1})h(x,b_{1})h(x,c_{1})\sqrt{1-x^{2}}}
 \frac{h(y,1)h(y,\sqrt{q})h(y,-1)h(y,-\sqrt{q})}{h(y,a_{2})h(y,b_{2})h(y,c_{2})h(y,d_{2})\sqrt{1-y^{2}}},
\\
 p_{j,k}^{(\alpha,\beta,\gamma,\delta)}(x,y;a_{1},b_{1},c_{1},d_{1};a_{2},b_{2},c_{2},d_{2}|q)\\
\qquad{} = \frac{p_{j}(x;a_{1},b_{1},c_{1}q^{\alpha k+\beta},d_{1}q^{\gamma k+\delta}|q)p_{k}(y;a_{2},b_{2},c_{2},d_{2}|q)}{h(x,c_{1}q^{\alpha k+\beta})},
\\
 q_{j,k}^{(\alpha,\beta,\gamma,\delta)}(x,y;a_{1},b_{1},c_{1},d_{1};a_{2},b_{2},c_{2},d_{2}|q)\\
\qquad{} = \frac{p_{j}(x;a_{1},b_{1},c_{1}q^{\alpha k+\beta},d_{1}q^{\gamma k+\delta}|q)p_{k}(y;a_{2},b_{2},c_{2},d_{2}|q)}{h(x,d_{1}q^{\gamma k+\delta})},
\end{gather*}
and
\begin{gather*}
 w(x,y;a_{1},b_{1};a_{2},b_{2},c_{2},d_{2}|q)\\
\qquad{} = \frac{h(x,1)h(x,\sqrt{q})h(x,-1)h(x,-\sqrt{q})}{h(x,a_{1})h(x,b_{1})\sqrt{1-x^{2}}}
 \frac{h(y,1)h(y,\sqrt{q})h(y,-1)h(y,-\sqrt{q})}{h(y,a_{2})h(y,b_{2})h(y,c_{2})h(y,d_{2})\sqrt{1-y^{2}}},
\end{gather*}
where
\begin{gather*}
x=\cos\theta,\qquad y=\cos\varphi,\qquad\alpha,\beta\ge0,
\end{gather*}
then
\begin{gather*}
 \int_{-1}^{1}\int_{-1}^{1}u_{j,k}^{(\gamma,\delta)}(x,y;a_{1},b_{1},c_{1},d_{1};a_{2},b_{2},c_{2},d_{2}|q)
 v_{m,n}^{(\gamma,\delta)}(x,y;a_{1},b_{1},c_{1},d_{1};a_{2},b_{2},c_{2},d_{2}|q)\\
\qquad\quad{} \times w(x,y;a_{1},b_{1},c_{1};a_{2},b_{2},c_{2},d_{2}|q)dxdy\\
\qquad{} = \frac{4\pi^{2}(a_{2}b_{2}c_{2}d_{2}q^{2k};q)_{\infty}(a_{2}b_{2}c_{2}d_{2}q^{k-1};q)_{k}}{(q^{k+1},a_{2}b_{2}q^{k};q)_{\infty}
(a_{2}c_{2}q^{k},a_{2}d_{2}q^{k},b_{2}c_{2}q^{k},b_{2}d_{2}q^{k},c_{2}d_{2}q^{k};q)_{\infty}}\\
\qquad\quad{}
\times \frac{(a_{1}b_{1}c_{1}d_{1}q^{2j+\gamma k+\delta};q)_{\infty}}{(q^{j+1},a_{1}b_{1}q^{j},a_{1}c_{1}q^{j+ k},a_{1}d_{1}q^{j+\gamma k+\delta};q)_{\infty}}\\
\qquad\quad{}
\times
 \frac{(a_{1}b_{1}c_{1}d_{1}q^{j+\gamma k+\delta-1};q)_{j}\delta_{j,m}\delta_{k,n}}{(b_{1}c_{1}q^{j+ k},b_{1}d_{1}q^{j+\gamma k+\delta},c_{1}d_{1}q^{j+\gamma k+\delta};q)_{\infty}}.
\end{gather*}
and
\begin{gather*}
 \int_{-1}^{1}\int_{-1}^{1}p_{j,k}^{(\alpha,\beta,\gamma,\delta)}(x,y;a_{1},b_{1},c_{1},d_{1};a_{2},b_{2},c_{2},d_{2}|q)\\
\qquad\quad{}\times q_{m,n}^{(\alpha,\beta,\gamma,\delta)}(x,y;a_{1},b_{1},c_{1},d_{1};a_{2},b_{2},c_{2},d_{2}|q)
 w(x,y;a_{1},b_{1};a_{2},b_{2},c_{2},d_{2}|q)dxdy\\
\qquad{}= \frac{4\pi^{2}(a_{2}b_{2}c_{2}d_{2}q^{2k};q)_{\infty}(a_{2}b_{2}c_{2}d_{2}q^{k-1};q)_{k}}{(q^{k+1},a_{2}b_{2}q^{k};q)_{\infty}(a_{2}c_{2}q^{k},a_{2}d_{2}q^{k},b_{2}c_{2}q^{k},b_{2}d_{2}q^{k},c_{2}d_{2}q^{k};q)_{\infty}}\\
\qquad\quad{} \times \frac{(a_{1}b_{1}c_{1}d_{1}q^{2j+(\alpha+\gamma)k+\beta+\delta};q)_{\infty}}{(q^{j+1},a_{1}b_{1}q^{j},a_{1}c_{1}q^{j+\alpha k+\beta},a_{1}d_{1}q^{j+\gamma k+\delta};q)_{\infty}}\\
\qquad\quad{}
\times \frac{(a_{1}b_{1}c_{1}d_{1}q^{j+(\alpha+\gamma)k+\beta+\delta-1};q)_{j}\delta_{j,m}\delta_{k,n}}{(b_{1}c_{1}q^{j+\alpha k+\beta},b_{1}d_{1}q^{j+\gamma k+\delta},c_{1}d_{1}q^{j+(\alpha+\gamma)k+\beta+\delta};q)_{\infty}}.
\end{gather*}
\end{Theorem}

\begin{proof}
The details of the calculation are
\begin{gather*}
 \int_{-1}^{1}\int_{-1}^{1}p_{j,k}^{(\alpha,\beta,\gamma,\delta)}(x,y;a_{1},b_{1},c_{1},d_{1};a_{2},b_{2},c_{2},d_{2}|q)\\
\qquad\quad{}\times q_{m,n}^{(\alpha,\beta,\gamma,\delta)}(x,y;a_{1},b_{1},c_{1},d_{1};a_{2},b_{2},c_{2},d_{2}|q)
 w(x,y;a_{1},b_{1};a_{2},b_{2},c_{2},d_{2}|q)dxdy\\
 \qquad{}
= \int_{-1}^{1}p_{j}\big(x;a_{1},b_{1},c_{1}q^{\alpha k+\beta},d_{1}q^{\gamma k+\delta}|q\big)
 p_{m}\big(x;a_{1},b_{1},c_{1}q^{\alpha n+\beta},d_{1}q^{\gamma n+\delta}|q\big)\\
\qquad\quad{}
\times \frac{h(x,1)h(x,\sqrt{q})h(x,-1)h(x,-\sqrt{q})dx}{h(x,a_{1})h(x,b_{1})h(x,c_{1}q^{\alpha k+\beta})h(x,d_{1}q^{\gamma n+\delta})\sqrt{1-x^{2}}}\\
\qquad\quad{}
\times \int_{-1}^{1}p_{k}(y;a_{2},b_{2},c_{2},d_{2}|q)p_{n}(y;a_{2},b_{2},c_{2},d_{2}|q)
 w(y;a_{2},b_{2},c_{2},d_{2}|q)dy\\
\qquad{} = \frac{2\pi(a_{2}b_{2}c_{2}d_{2}q^{2k};q)_{\infty}(a_{2}b_{2}c_{2}d_{2}q^{k-1};q)_{k}\delta_{k,n}}{(q^{k+1};q)_{\infty}(a_{2}b_{2}q^{k},a_{2}c_{2}q^{k},a_{2}d_{2}q^{k},b_{2}c_{2}q^{k},b_{2}d_{2}q^{k},c_{2}d_{2}q^{k};q)_{\infty}}\\
\qquad\quad{}\times \int_{-1}^{1}p_{j}(x;a_{1},b_{1},c_{1}q^{\alpha k+\beta},d_{1}q^{\gamma k+\delta}|q)p_{m}(x;a_{1},b_{1},c_{1}q^{\alpha k+\beta},d_{1}q^{\gamma k+\delta}|q)\\
\qquad\quad{}
\times \frac{h(x,1)h(x,\sqrt{q})h(x,-1)h(x,-\sqrt{q})dx}{h(x,a_{1})h(x,b_{1})h(x,c_{1}q^{\alpha k+\beta})h(x,d_{1}q^{\gamma k+\delta})\sqrt{1-x^{2}}}\\
\qquad{}
= \frac{4\pi^{2}(a_{2}b_{2}c_{2}d_{2}q^{2k};q)_{\infty}(a_{2}b_{2}c_{2}d_{2}q^{k-1};q)_{k}}{(q^{k+1},a_{2}b_{2}q^{k};q)_{\infty}(a_{2}c_{2}q^{k},a_{2}d_{2}q^{k},b_{2}c_{2}q^{k},b_{2}d_{2}q^{k},c_{2}d_{2}q^{k};q)_{\infty}}\\
\qquad\quad{}
\times \frac{(a_{1}b_{1}c_{1}d_{1}q^{2j+(\alpha+\gamma)k+\beta+\delta};q)_{\infty}}{(q^{j+1},a_{1}b_{1}q^{j},a_{1}c_{1}q^{j+\alpha k+\beta},a_{1}d_{1}q^{j+\gamma k+\delta};q)_{\infty}}\\
\qquad\quad{}
\times \frac{(a_{1}b_{1}c_{1}d_{1}q^{j+(\alpha+\gamma)k+\beta+\delta-1};q)_{j}\delta_{j,m}\delta_{k,n}}{(b_{1}c_{1}q^{j+\alpha k+\beta},b_{1}d_{1}q^{j+\gamma k+\delta},c_{1}d_{1}q^{j+(\alpha+\gamma)k+\beta+\delta};q)_{\infty}}.
\end{gather*}
The $u$, $v$ orthogonality can be proved similarly.
 \end{proof}

\begin{Theorem}
\label{theorem:tensor-askey-wilson-1} Let
\begin{gather*}
 p_{j,k}^{(\gamma,\delta)}(x,y;a_{1},b_{1},c_{1},d_{1};a_{2},b_{2},c_{2},d_{2}|q)
= \frac{p_{j}(x;a_{1},b_{1},c_{1},d_{1}q^{\gamma k+\delta}|q)p_{k}(y;a_{2},b_{2},c_{2},d_{2}|q)}{(d_{1}q^{\gamma k+\delta}e^{i\theta};q)_{\infty}}.
\end{gather*}
Then
\begin{gather*}
 \int_{-1}^{1}\int_{-1}^{1}p_{j,k}^{(\gamma,\delta)}(x,y;a_{1},b_{1},c_{1},d_{1};a_{2},b_{2},c_{2},d_{2}|q)\\
\qquad\quad{}\times \overline{p_{m,n}^{(\gamma,\delta)}(x,y;a_{1},b_{1},c_{1},d_{1};a_{2},b_{2},c_{2},d_{2}|q)}
 w(x,y;a_{1},b_{1},c_{1};a_{2},b_{2},c_{2},d_{2}|q)dxdy\\
\qquad{}
= \frac{4\pi^{2}\delta_{j,m}\delta_{k,n}}{(q^{m+1},a_{1}b_{1}q^{m},a_{1}c_{1}q^{m},b_{1}c_{1}q^{m},b_{1}c_{1}q^{m};q)_{\infty}}\\
\qquad\quad{}
\times \frac{(a_{1}b_{1}c_{1}d_{1}q^{2m+\gamma n+\delta};q)_{\infty}(a_{1}b_{1}c_{1}d_{1}q^{m+\gamma n+\delta-1};q)_{m}}{(a_{1}d_{1}q^{m+\gamma n+\delta},b_{1}d_{1}q^{m+\gamma n+\delta},c_{1}d_{1}q^{m+\gamma n+\delta};q)_{\infty}}\\
\qquad\quad{} \times \frac{(a_{2}b_{2}c_{2}d_{2}q^{2n};q)_{\infty}(a_{2}b_{2}c_{2}d_{2}q^{n-1};q)_{n}}{(q^{n+1},a_{2}
b_{2}q^{n},a_{2}c_{2}q^{n},a_{2}d_{2}q^{n},b_{2}c_{2}q^{n},b_{2}d_{2}q^{n},c_{2}d_{2}q^{n};q)_{\infty}}.
\end{gather*}
where
$x=\cos\theta$, $y=\cos\varphi$, $\alpha,\beta\ge0$ and $w(x,y;a_{1},b_{1},c_{1};a_{2},b_{2},c_{2},d_{2}|q)$ is the same as in Theorem~{\rm \ref{theorem:tensor-askey-wilson}}.
\end{Theorem}

\begin{proof}
Observe that the integral
\begin{gather*}
\int_{-1}^{1}\int_{-1}^{1}p_{j,k}^{(\gamma,\delta)}(x,y;a_{1},b_{1},c_{1},d_{1};a_{2},b_{2},c_{2},d_{2}|q) w(x,y;a_{1},b_{1},c_{1};a_{2},b_{2},c_{2},d_{2}|q \\
\qquad{} \times \overline{p_{m,n}^{(\gamma,\delta)}(x,y;a_{1},b_{1},c_{1},d_{1};a_{2},b_{2},c_{2},d_{2}|q)}) dxdy
\end{gather*}
equals
\begin{gather*}
 \int_{-1}^{1}p_{j}\big(x;a_{1},b_{1},c_{1},d_{1}q^{\gamma k+\delta}|q\big) p_{m}\big(x;a_{1},b_{1},c_{1},d_{1}q^{\gamma n+\delta}|q\big)\\
 \qquad\quad{}
\times \frac{h(x,1)h(x,\sqrt{q})h(x,-1)h(x,-\sqrt{q})dx}{h(x,a_{1})h(x,b_{1})h(x,c_{1})(d_{1}q^{\gamma k+\delta}e^{i\theta},d_{1}q^{\gamma n+\delta}e^{-i\theta};q)_{\infty}\sqrt{1-x^{2}}}\\
\qquad\quad{}
\times \int_{-1}^{1}p_{k}(y;a_{2},b_{2},c_{2},d_{2}|q)p_{n}(y;a_{2},b_{2},c_{2},d_{2}|q)
 w(y;a_{2},b_{2},c_{2},d_{2}|q)dy\\
 \qquad{}
= \frac{2\pi\delta_{k,n}(a_{2}b_{2}c_{2}d_{2}q^{2n};q)_{\infty}(a_{2}b_{2}c_{2}d_{2}q^{n-1};q)_{n}}{(q^{n+1};q)_{\infty}(a_{2}b_{2}q^{n},a_{2}c_{2}q^{n},a_{2}d_{2}q^{n},b_{2}c_{2}q^{n},b_{2}d_{2}q^{n},c_{2}d_{2}q^{n};q)_{\infty}}\\
\qquad\quad{}
\times \int_{-1}^{1}p_{j}\big(x;a_{1},b_{1},c_{1},d_{1}q^{\gamma n+\delta}|q\big)p_{m}\big(x;a_{1},b_{1},c_{1},d_{1}q^{\gamma n+\delta}|q\big)\\
\qquad\quad{}
\times \frac{h(x,1)h(x,\sqrt{q})h(x,-1)h(x,-\sqrt{q})dx}{h(x,a_{1})h(x,b_{1})h(x,c_{1})h(x,d_{1}q^{\gamma n+\delta})\sqrt{1-x^{2}}}\\
\qquad{}
= \frac{4\pi^{2}\delta_{j,m}\delta_{k,n}}{(q^{m+1},a_{1}b_{1}q^{m},a_{1}c_{1}q^{m},b_{1}c_{1}q^{m},b_{1}c_{1}q^{m};q)_{\infty}}\\
\qquad\quad{}
\times \frac{(a_{1}b_{1}c_{1}d_{1}q^{2m+\gamma n+\delta};q)_{\infty}(a_{1}b_{1}c_{1}d_{1}q^{m+\gamma n+\delta-1};q)_{m}}{(a_{1}d_{1}q^{m+\gamma n+\delta},b_{1}d_{1}q^{m+\gamma n+\delta},c_{1}d_{1}q^{m+\gamma n+\delta};q)_{\infty}}\\
\qquad\quad{}
\times \frac{(a_{2}b_{2}c_{2}d_{2}q^{2n};q)_{\infty}(a_{2}b_{2}c_{2}d_{2}q^{n-1};q)_{n}}{(q^{n+1},a_{2}b_{2}q^{n},
a_{2}c_{2}q^{n},a_{2}d_{2}q^{n},b_{2}c_{2}q^{n},b_{2}d_{2}q^{n},c_{2}d_{2}q^{n};q)_{\infty}}.\tag*{\qed}
\end{gather*}
\renewcommand{\qed}{}
 \end{proof}

\subsection*{Acknowledgements}

The authors are very grateful to the anonymous referees for their detailed reports on the f\/irst draft of this paper.
 Research of M.E.H.~Ismail supported by a grant from DSFP program at King Saud and by the National Plan for Science, Technology and innovation (MAARIFAH), King Abdelaziz City for Science and Technology, Kingdom of Saudi Arabia, Award number 14-MAT623-02.
 Research of R.~Zhang partially supported by National Science Foundation of China, grant No.~11371294.

\pdfbookmark[1]{References}{ref}
\LastPageEnding

\end{document}